\documentclass[reqno]{amsart} %REQNO = For left-aligned equation enumeration.
\usepackage[headings]{fullpage}
\usepackage{amsmath} 
\usepackage{paralist}
\usepackage{amssymb} 
\usepackage{amsthm}
\usepackage{amsfonts}
\usepackage[centertags]{amsmath}
\usepackage{fullpage}
\usepackage{comment} 
\usepackage{bm}
\usepackage{thmtools}
\usepackage{thm-restate}
\usepackage{hyperref, float}
\hypersetup{colorlinks=true,citecolor=purple,linkcolor=purple}

\usepackage[backrefs]{amsrefs}

\usepackage{tikz-cd}
\usepackage{graphicx}
\usepackage{mathrsfs}
\usepackage[export]{adjustbox} %To add frames to figures.
\usepackage{xcolor}
\usepackage[margin=1.25in]{geometry}
\usepackage{amscd}
\usetikzlibrary{shapes}
\usetikzlibrary{arrows}
\usetikzlibrary{positioning}

\theoremstyle{definition}
\newtheorem{theorem}{Theorem}[section]

\newtheorem{lemma}[theorem]{Lemma}
\newtheorem{proposition}[theorem]{Proposition}
\newtheorem{definition}[theorem]{Definition}
\newtheorem{corollary}[theorem]{Corollary}

\theoremstyle{definition}
\newtheorem*{thm*}{Theorem}

\newtheorem{mydef}[theorem]{Definition}
\newtheorem{lem}[theorem]{Lemma}

\newtheorem{example}[theorem]{Example}

\newtheorem{remark}[theorem]{Remark}

\newcommand{\Ass}{\text{Ass}}

\newcommand{\Min}{\text{Min}}

\DeclareMathOperator{\Spec}{Spec}

\DeclareMathOperator{\height}{ht}

\DeclareMathOperator{\vp}{\varphi}

\DeclareMathOperator{\id}{id}

\begin{document}

\title[Spectra of Noetherian UFDs]{Every Finite Poset is Isomorphic to a Saturated Subset of the Spectrum of a Noetherian UFD}
\author{C. Colbert and S. Loepp}

\thanks{The first author was partially supported by a Lenfest grant from Washington and Lee University.}

\maketitle

\begin{abstract}
We show that every finite poset is isomorphic to a saturated subset of the spectrum of a Noetherian unique factorization domain.  In addition, we show that every finite poset is isomorphic to a saturated subset of the spectrum of a quasi-excellent domain.
\end{abstract}

\tikzstyle{main} = [draw, circle, draw=black, fill=black, minimum size = 1mm]
\tikzstyle{main2} = [draw, circle, draw=blue, fill=blue, minimum size = 1mm]
\tikzstyle{main3} = [draw, circle, draw=red, fill=red, minimum size = 1mm]
\tikzstyle{main4} = [draw, circle, draw=white, fill=white, minimum size = 1mm]

\section{Introduction}

%Although commutative rings are a central object of study within commutative algebra, there still remain tantalizing mysteries regarding the basic anatomy of important classes of commutative rings. 
Although commutative rings are quite possibly the most important object in commutative algebra, there remain tantalizing mysteries regarding their anatomy.  In particular, the prime ideal structure is not well understood for large classes of commutative rings.  In this article, we focus on better understanding the set of prime ideals of a commutative ring -- referred to as the prime spectrum of the ring -- considered as a partially ordered set with respect to inclusion.
%In particular, the prime ideal structure -- \textcolor{blue}{referred to as the spectrum of the ring} -- is not well understood for large classes of commutative rings.  
On the positive side, Hochster proved a remarkable result in 1969 \cite{Hochster}.  Given a partially ordered set $X$, he found necessary and sufficient conditions for the existence of a commutative ring $R$ such that $X$ is isomorphic to the prime spectrum of $R$ as partially ordered sets. The analogous result for Noetherian rings, however, remains wide open.  That is, given a partially ordered set $X$, it is unknown exactly when there exists a Noetherian ring $R$ such that $X$ is isomorphic to the prime spectrum of $R$.  There has been progress made (see, for example, \cite{WiegandEXPO} for a nice survey on the topic), but the general question is still unsolved, even when the dimension of the partially ordered set is two.  

Given that this open question seems so elusive even more than 50 years after Hochster's beautiful result, it is reasonable to ask the related but less onerous question: ``Which finite partially ordered sets can occur as part of the prime spectrum of a Noetherian ring?" To be more precise, we ask, ``Which finite partially ordered sets can be embedded into the prime spectrum of a Noetherian ring in a way that preserves saturated chains?" We provide a rigorous definition of this notion in Section \ref{Prelims}, and when this happens, we say that the given poset is isomorphic to a saturated subset of the prime spectrum of the ring.  Finite partially ordered sets that are not catenary are of particular interest.  In the early twentieth century, it was thought that, since a local (Noetherian) domain is a dense subspace of its completion, which is necessarily catenary, all local domains might be catenary.  In 1956, however, Nagata constructed a family of  noncatenary local domains \cite{Nagata}.  More than 20 years later, Heitmann in \cite{HeitmannNoncatenary} showed that, perhaps surprisingly, given any finite partially ordered set $X$, there is a Noetherian domain $R$ such that $X$ is isomorphic to a saturated subset of the prime spectrum of $R$.  In other words, {\em any} noncatenary finite poset can occur as part of the prime spectrum of a Noetherian domain.  

A natural extension is to ask which finite partially ordered sets can be embedded into the prime spectrum of a Noetherian ring with a particularly nice property.  For example, the family of noncatenary Noetherian domains that Nagata constructed enjoy the property that their integral closures are catenary.  So in 1956, Nagata asked the natural question of whether or not integrally closed domains must be catenary.  In a 1980 paper \cite{Ogoma}, Ogoma answered this question by constructing noncatenary integrally closed local domains.  Given this, one might now wonder how ``nice" noncatenary rings can be.  For example, regular local rings and excellent rings are necessarily catenary.  As Noetherian unique factorization domains sit between integrally closed domains (which can be noncatenary) and regular local rings (which must be catenary), a next reasonable question is: do there exist noncatenary Noetherian unique factorization domains? This question was not answered until 1993 when Heitmann in \cite{heitmannUFD} constructed a noncatenary unique factorization domain. This leads to the natural question: can {\em any} noncatenary finite poset occur as part of the prime spectrum of a Notherian unique factorization domain? Or, more precisely, given a finite partially ordered set $X$, does there exist a Noetherian unique factorization domain $A$ such that $X$ is isomorphic to a saturated subset of the prime spectrum of $A$?  Partial progress on this problem is made in \cite{Avery} and \cite{Semendinger} where it is shown that certain noncatenary partially ordered sets are isomorphic to a saturated subset of the prime spectrum of a Noetherian unique factorization domain.  In this article, we find a definitive answer.  In a result we find rather surprising (Theorem \ref{FinalUFDTheorem}), we show that {\em all} finite partially ordered sets occur as part of the prime spectrum of a Noetherian unique factorization domain. In other words, we show that, given a finite partially ordered set $X$, there is a Noetherian unique factorization domain $A$ such that $X$ is isomorphic to a saturated subset of the prime spectrum of $A$. This shows that rings with very desirable properties can have very badly behaved prime spectra, and, in particular, that there are no restrictions on finite subsets of the prime spectra of Noetherian unique factorization domains.  We note that our result generalizes the result in \cite{HeitmannNoncatenary} discussed in the previous paragraph.

We also consider quasi-excellent rings, another class of Noetherian rings that behave well, and we show that an analogous result holds.  In other words, we show that, given any finite  partially ordered set $X$, there is a quasi-excellent domain $B$ satisfying the property that there is an embedding from $X$ to the prime spectrum of $B$ that preserves saturated chains.  Specifically, in Corollary \ref{FinalQuasiExcellent}, we show that given a finite partially ordered set $X$, there exists a quasi-excellent domain $B$ such that $X$ is isomorphic to a saturated subset of the prime spectrum of $B$.  This shows that, even though $B$ is a ring with geometrically regular formal fibers, it can have a prime spectrum containing a finite subset that is arbitrarily badly behaved. Notice that the result cannot possibly hold for excellent rings since excellent rings are catenary by definition.  Relatedly, it is worth noting that, by choosing finite partially ordered sets that are not catenary, our result provides a large class of rings that are quasi-excellent but not excellent.  Since a Noetherian ring is excellent if it is a G-ring, is J-2, and is universally catenary, our results show that the universally catenary condition is truly disjoint from the other two conditions.  

%\textcolor{red}{put something in here about the referee comment: I am not sure where to work this in, but there is a very interesting aspect to this article which merits inclusion somewhere. A ring is excellent if it is a G-ring, is J-2, and is universally catenarian. This paper demonstrates that the catenarian condition is truly disjoint from the other two.}

To illustrate how we construct our rings, we consider a specific example, and we informally describe our construction for that example. All of the ideas we describe in the example will be formalized later in the paper.   Let $X$ be the finite partially ordered set pictured in Figure \ref{DrawingOfX}, and we note that previously, it was unknown whether or not there exists a Noetherian unique factorization domain $A$ such that $X$ is isomorphic to a saturated subset of the prime spectrum of $A$.  Note that $X$ can be realized using a sequence of partially ordered sets where we start with a partially ordered set containing one element and, to get from one partially ordered set to the next, we either ``grow" nodes out of a minimal node, or we ``glue" two minimal nodes together.  This sequence of partially ordered sets is illustrated in Figures \ref{GrowingGluingPicture1} and \ref{GrowingGluingPicture2}.  We show that, for each partially ordered set in the sequence, there is a corresponding local ring satisfying some nice properties such that the given partially ordered set is isomorphic to a saturated subset of the prime spectrum of the local ring.  We also show that, if a local ring corresponding to one partially ordered set in our sequence is quasi-excellent, then so is the local ring corresponding to the next partially ordered set in the sequence.  We use $\mathbb{C}$ as our first ring, and so, since $\mathbb{C}$ is quasi-excellent, our last local ring will be as well.  It follows that the last local ring in our sequence is quasi-excellent and satisfies the property that its prime spectrum contains a saturated subset that is isomorphic to $X$.  Finally, we use the nice properties of our local rings to show that there is a Noetherian unique factorization domain $A$ such that $X$ is isomorphic to a saturated subset of the prime spectrum of $A$.

%%%%%%%%%%%%%%%%%%%%%%%%%%%%%%%%%
\begin{figure}
    \centering
    \begin{tikzpicture}
%NODES
\node[main] (1) {};
\node[main] (4) [below right of=1] {};
\node[main] (5) [below of=4] {};
\node[main] (6) [below left of= 5] {};
\node[main] (2) [below left of=4, left of=4] {};

\node[main] (8) [below right of=6] {};
\node[main] (9) [below of=8] {};
\node[main] (10) [below left of= 9] {};
\node[main] (7) [below left of=8, left of=8] {};

%DRAWING
\draw (1) -- (4) -- (5) -- (6) -- (8) -- (9) -- (10);
\draw (1) -- (2) -- (6) -- (7) -- (10);
\end{tikzpicture}
    \caption{Poset $X$}
    \label{DrawingOfX}
\end{figure}

%%%%%%%%TOP%%%%%%%%%%%
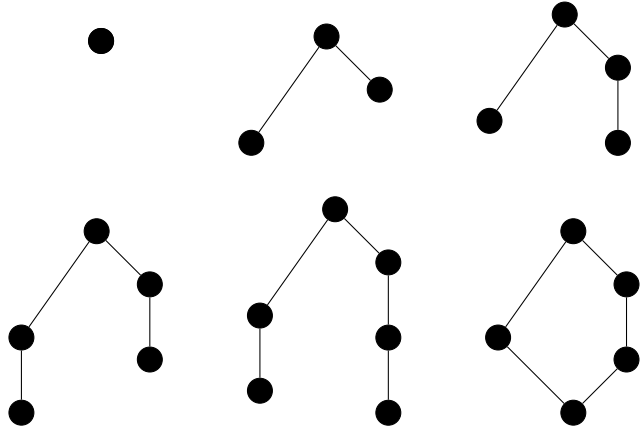
\begin{figure}
    \centering
\begin{tikzpicture}

\node[main] (1) {};
%NODES
\node[main] (1) {};
\node[white] (4) [below right of=1] {};
%\node[main] (5) [below of=4] {5};
%\node[main] (6) [below left of= 5] {6};
\node[white] (2) [below left of=4, left of=4] {};
%\node[main] (6L) [below of = 2] {$6'$};
%\node[main] (6R) [below of = 5] {$6''$};

%\node[main] (8) [below right of=6] {8};
%\node[main] (9) [below of=8] {9};
%\node[main] (10) [below left of= 9] {10};
%\node[main] (7) [below left of=8, left of=8] {7};
%\node[main] (10L) [below of = 7] {$10'$};
%\node[main] (10R) [below of =9] {$10''$};

%DRAWING
%\draw (1) -- (4);
%\draw (1) -- (2);
\end{tikzpicture}\hspace{1cm}\vspace{.5cm}
\begin{tikzpicture} 
%NODES
\node[main] (1) {};
\node[main] (4) [below right of=1] {};
%\node[main] (5) [below of=4] {5};
%\node[main] (6) [below left of= 5] {6};
\node[main] (2) [below left of=4, left of=4] {};
%\node[main] (6L) [below of = 2] {$6'$};
%\node[main] (6R) [below of = 5] {$6''$};

%\node[main] (8) [below right of=6] {8};
%\node[main] (9) [below of=8] {9};
%\node[main] (10) [below left of= 9] {10};
%\node[main] (7) [below left of=8, left of=8] {7};
%\node[main] (10L) [below of = 7] {$10'$};
%\node[main] (10R) [below of =9] {$10''$};

%DRAWING
\draw (1) -- (4);
\draw (1) -- (2);
\end{tikzpicture}\hspace{1cm}
\begin{tikzpicture}
%NODES
\node[main] (1) {};
\node[main] (4) [below right of=1] {};
\node[main] (5) [below of=4] {};
%\node[main] (6) [below left of= 5] {6};
\node[main] (2) [below left of=4, left of=4] {};
%\node[main] (6L) [below of = 2] {$6'$};
%\node[main] (6R) [below of = 5] {$6''$};

%\node[main] (8) [below right of=6] {8};
%\node[main] (9) [below of=8] {9};
%\node[main] (10) [below left of= 9] {10};
%\node[main] (7) [below left of=8, left of=8] {7};
%\node[main] (10L) [below of = 7] {$10'$};
%\node[main] (10R) [below of =9] {$10''$};

%DRAWING
\draw (1) -- (4) -- (5);
\draw (1) -- (2);
\end{tikzpicture}

\begin{tikzpicture}
%NODES
\node[main] (1) {};
\node[main] (4) [below right of=1] {};
\node[main] (5) [below of=4] {};
%\node[main] (6) [below left of= 5] {6};
\node[main] (2) [below left of=4, left of=4] {};
\node[main] (6L) [below of = 2] {};
%\node[main] (6R) [below of = 5] {$6''$};

%\node[main] (8) [below right of=6] {8};
%\node[main] (9) [below of=8] {9};
%\node[main] (10) [below left of= 9] {10};
%\node[main] (7) [below left of=8, left of=8] {7};
%\node[main] (10L) [below of = 7] {$10'$};
%\node[main] (10R) [below of =9] {$10''$};

%DRAWING
\draw (1) -- (4) -- (5);
\draw (1) -- (2) -- (6L);
\end{tikzpicture}\hspace{1cm}
\begin{tikzpicture}
%NODES
\node[main] (1) {};
\node[main] (4) [below right of=1] {};
\node[main] (5) [below of=4] {};
\node[main] (2) [below left of=4, left of=4] {};
\node[main] (6L) [below of = 2] {};
\node[main] (6R) [below of = 5] {};

%\node[main] (8) [below right of=6] {8};
%\node[main] (9) [below of=8] {9};
%\node[main] (10) [below left of= 9] {10};
%\node[main] (7) [below left of=8, left of=8] {7};
%\node[main] (10L) [below of = 7] {$10'$};
%\node[main] (10R) [below of =9] {$10''$};

%DRAWING
\draw (1) -- (4) -- (5) -- (6R);
\draw (1) -- (2) -- (6L);
\end{tikzpicture}\hspace{1cm}
\begin{tikzpicture}
%NODES
\node[main] (1) {};
\node[main] (4) [below right of=1] {};
\node[main] (5) [below of=4] {};
\node[main] (6) [below left of= 5] {};
\node[main] (2) [below left of=4, left of=4] {};

%\node[main] (8) [below right of=6] {8};
%\node[main] (9) [below of=8] {9};
%\node[main] (10) [below left of= 9] {10};
%\node[main] (7) [below left of=8, left of=8] {7};
%\node[main] (10L) [below of = 7] {$10'$};
%\node[main] (10R) [below of =9] {$10''$};

%DRAWING
\draw (1) -- (4) -- (5) -- (6);
\draw (1) -- (2) -- (6);
\end{tikzpicture}
    \caption{Growing and Gluing to Obtain $X$}
    \label{GrowingGluingPicture1}
\end{figure}%%%%%%%%%%%%%%%%%%%%%END FIGURE 1%%%%%%%%%%%%%%%%%%%%%%%
\begin{figure}
    \centering
\begin{tikzpicture}
%NODES
\node[main] (1) {};
\node[main] (4) [below right of=1] {};
\node[main] (5) [below of=4] {};
\node[main] (6) [below left of= 5] {};
\node[main] (2) [below left of=4, left of=4] {};

\node[main] (8) [below right of=6] {};
%\node[main] (9) [below of=8] {9};
%\node[main] (10) [below left of= 9] {10};
\node[main] (7) [below left of=8, left of=8] {};
%\node[main] (10L) [below of = 7] {$10'$};
%\node[main] (10R) [below of =9] {$10''$};

%DRAWING
\draw (1) -- (4) -- (5) -- (6) -- (8);
\draw (1) -- (2) -- (6) -- (7);
\end{tikzpicture}\hspace{2cm}
\begin{tikzpicture}
%NODES
\node[main] (1) {};
\node[main] (4) [below right of=1] {};
\node[main] (5) [below of=4] {};
\node[main] (6) [below left of= 5] {};
\node[main] (2) [below left of=4, left of=4] {};

\node[main] (8) [below right of=6] {};
\node[main] (9) [below of=8] {};
%\node[main] (10) [below left of= 9] {10};
\node[main] (7) [below left of=8, left of=8] {};
%\node[main] (10L) [below of = 7] {$10'$};
%\node[main] (10R) [below of =9] {$10''$};

%DRAWING
\draw (1) -- (4) -- (5) -- (6) -- (8) -- (9);
\draw (1) -- (2) -- (6) -- (7);
\end{tikzpicture}\hspace{2cm}
\begin{tikzpicture}
%NODES
\node[main] (1) {};
\node[main] (4) [below right of=1] {};
\node[main] (5) [below of=4] {};
\node[main] (6) [below left of= 5] {};
\node[main] (2) [below left of=4, left of=4] {};

\node[main] (8) [below right of=6] {};
\node[main] (9) [below of=8] {};
%\node[main] (10) [below left of= 9] {10};
\node[main] (7) [below left of=8, left of=8] {};
\node[main] (10L) [below of = 7] {};
%\node[main] (10R) [below of =9] {$10''$};

%DRAWING
\draw (1) -- (4) -- (5) -- (6) -- (8) -- (9);
\draw (1) -- (2) -- (6) -- (7) -- (10L);
\end{tikzpicture}

\begin{tikzpicture}
%NODES
\node[main] (1) {};
\node[main] (4) [below right of=1] {};
\node[main] (5) [below of=4] {};
\node[main] (6) [below left of= 5] {};
\node[main] (2) [below left of=4, left of=4] {};

\node[main] (8) [below right of=6] {};
\node[main] (9) [below of=8] {};
%\node[main] (10) [below left of= 9] {10};
\node[main] (7) [below left of=8, left of=8] {};
\node[main] (10L) [below of = 7] {};
\node[main] (10R) [below of =9] {};

%DRAWING
\draw (1) -- (4) -- (5) -- (6) -- (8) -- (9) -- (10R);
\draw (1) -- (2) -- (6) -- (7) -- (10L);
\end{tikzpicture} \hspace{2cm}
\begin{tikzpicture}
%NODES
\node[main] (1) {};
\node[main] (4) [below right of=1] {};
\node[main] (5) [below of=4] {};
\node[main] (6) [below left of= 5] {};
\node[main] (2) [below left of=4, left of=4] {};

\node[main] (8) [below right of=6] {};
\node[main] (9) [below of=8] {};
\node[main] (10) [below left of= 9] {};
\node[main] (7) [below left of=8, left of=8] {};

%DRAWING
\draw (1) -- (4) -- (5) -- (6) -- (8) -- (9) -- (10);
\draw (1) -- (2) -- (6) -- (7) -- (10);
\end{tikzpicture}
    \caption{Growing and Gluing to Obtain $X,$ continued}
    \label{GrowingGluingPicture2}
\end{figure}
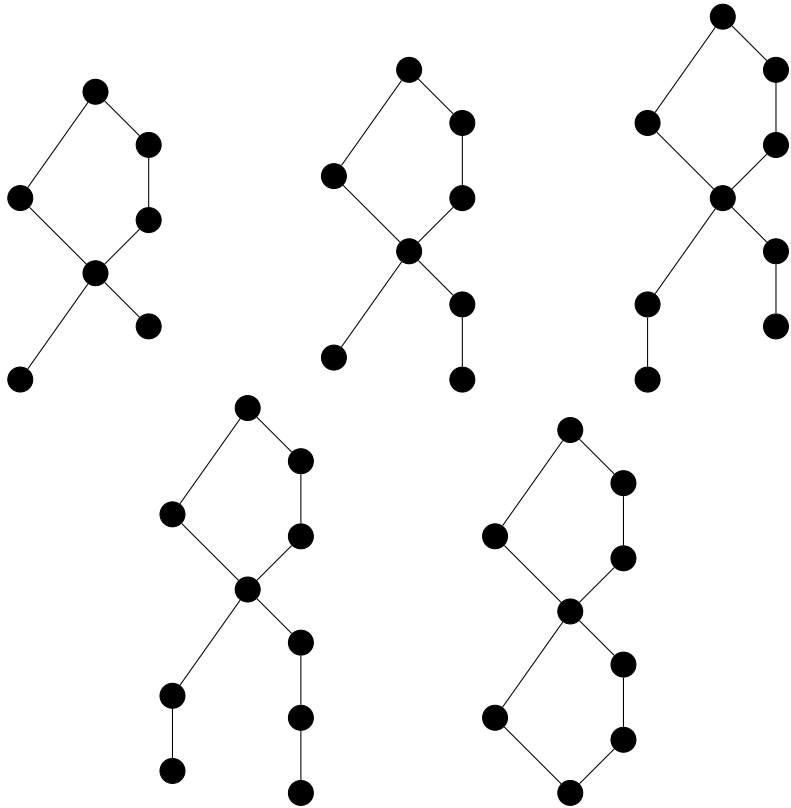
%%%%%%%%%%%%%%%%%%%%%%%%%%%%%%%%%%%%%%%%%

In Section \ref{Prelims} we find sufficient conditions for a partially ordered set to be isomorphic to a saturated subset of the prime spectrum of a ring.  In Section \ref{Growing and Gluing}, we describe how to find rings that correspond to our growing and gluing process described in the previous paragraph, and in Section \ref{UFDTheorem}, we show how to construct our final Noetherian unique factorization domain.  To show our result in full generality, it is necessary to show that every finite partially ordered set can be realized with our growing and gluing process.  This is delicate and technical and is done in Section \ref{GluingNodes}.  Our main results are in Section \ref{MainResults}, with our two central results being Corollary \ref{FinalQuasiExcellent} and Theorem \ref{FinalUFDTheorem}.

%We start with the single top node of $X$ and the ring $\mathbb{C}$ whose prime spectrum has exactly one element.

%and let $Z$ Let $B$ be a local ring satisfying some desirable conditions that we will specify later, and let $\{P_1, P_2, \ldots ,P_m\}$ be the minimal prime ideals of $B$.

%In Section \ref{Growing and Gluing}, we first show that, given a local ring satisfying some desirable conditions, there is another  

%Along a different axis of ``nice", we consider excellent.

%In this article, we consider the prime ideal structure of two classes of  Noetherian rings that are well behaved.  In particular, we focus on Noetherian unique factorization domains and quasi-excellent rings.

%The main result of this article is the following.  Let $X$ be a finite poset.  Then there exists a Noetherian UFD $A$ such that $X$ is isomorphic to a saturated subset of $\Spec(A)$.  As a consequence, every finite poset is isomorphic to a saturated subset of the spectrum of a Noetherian UFD.

%Things to put in intro:

%Strengthens the results of Heitmann, and some others (SMALL 2017, Alex paper).

%Along the way, we prove results that are interesting in their own right (gluing thm and UFD thm).

%Nagata's conjecture that normal rings were catenary?  (look that up)

%History -- we didn't even know there was a noncatenary UFD until the early 90's.  Now we show there are no restrictions on finite subsets of spectra of Noetherian UFDs.

%Explain the procedure -- growing and gluing and then, at the last step, a UFD.

\vspace{.2cm}

\noindent {\bf Notation.} All rings in this paper are commutative with unity.  When we say a ring is local, Noetherian is implied, and we use quasi-local for a ring with one maximal ideal that is not necessarily Noetherian.  If $B$ is a local ring with maximal ideal $M$ we use $(B,M)$ to denote the ring. If $(B,M)$ is a local ring, we use $\widehat{B}$ to denote the $M$-adic completion of $B$.  We use the standard abbreviations UFD for a unique factorization domain and RLR for a regular local ring, and we use the term poset for a partially ordered set.  Finally, when we say the spectrum of a ring $B$ or we write $\Spec{(B)}$, we mean the prime spectrum of $B$, and $\Min{(B)}$ denotes the minimal prime ideals of $B$.

%talk about quasi-excellent rings, and that maybe we have a large class of them.

%Put in an example to illustrate our procedure -- maybe the ``hour glass" noncatenary picture.

%Local implies Noetherian.

%We use $\widehat{B}$ to denote the completion of a local ring $B$ at its maximal ideal.

%Notation $(B,M)$

%UFD means unique factorization domain.

%RLR means regular local ring.

%poset means partially ordered set.

%talk about the quasi-excellent stuff.  Note that the ``quasi" part is absolutely necessary since excellent rings are catenary.

%put in all abbreviations

%put in structure of the paper -- what we do in each section.

\section{Preliminaries} \label{Prelims}

In this section, we first define what it means for a subposet $Z$ of a poset $Y$ to be a saturated subset of $Y$.  Then, given two posets $X$ and $Y$, we present sufficient conditions for $X$ to be isomorphic to a saturated subset of $Y$.

\begin{definition}
Let $X$ and $Y$ be posets. We say $f: X \to Y$ is a \textit{poset map} if for all $x, y \in X,$ $x \le y$ implies $f(x) \le f(y).$ We say a poset map $f:  X \to Y$ is a \textit{poset embedding} if for all $x, y \in X,$ $f(x) \le f(y)$ implies $x \le y.$ A surjective poset embedding from $X$ onto $Y$ is called a \textit{poset isomorphism.}
\end{definition}

\begin{remark}
Note that poset embeddings are necessarily injective: if $f: X \to Y$ is a poset embedding and $f(x) = f(y),$ then $f(x) \le f(y)$ and $f(y) \le f(x).$ So $x \le y$ and $y \le x,$ and therefore $x = y.$
\end{remark}

%\begin{definition}
%Let $X$ be a poset, and let $u,v \in X$.  We say that $v$ covers $u$ if $u < v$ and $w \in %X$ such that $u \leq w \leq v$ implies that $w = u$ or $w = v$.
%\end{definition}

\begin{definition}\label{saturatedDefinition}
Let $X$ be a poset, and let $x, y \in X.$ We say $y$ \textit{covers} $x$, and we write $x <_c y$, if $x < y$ and for all $z \in X,$ if $x \le z \le y,$ then $x = z$ or $y = z.$ We say $f: X \to Y$ is a \textit{saturated embedding} if $f$ is a poset embedding and for all $x, y \in X,$ if $y$ covers $x,$ then $f(y)$ covers $f(x)$ in $Y.$ If $Z \subseteq Y$ is a subposet (i.e., $Z$ is a poset under the same order relation on $Y$) we say $Z$ is a \textit{saturated} subset of $Y$ if for all $u, v \in Z,$ whenever $v$ covers $u$ in $Z,$ it also covers $u$ in $Y.$
\end{definition}

\begin{remark}
If $f: X \to Y$ is a saturated embedding, then $y$ covers $x$ if and only if $f(y)$ covers $f(x).$ To see this, suppose $f(y)$ covers $f(x).$ Then $f(x) < f(y),$ which implies $x < y$ since $f$ is a poset embedding. If $x \le z \le y$ for $z \in X,$ then $f(x) \le f(z) \le f(y),$ so $f(x) = f(z)$ or $f(y) = f(z).$ Since every poset embedding is an injective map, $x = z$ or $y = z.$
\end{remark}

\begin{definition}\label{completeDefinition}
If $C \subseteq Y$ is a subposet of $Y,$ we say $C$ is a \textit{complete} subset of $Y$ if for all $u, v \in C,$ if $u \le y \le v$ for some $y \in Y,$ then $y \in C.$
\end{definition}

\begin{remark}
If $Z$ is a complete subset of $Y,$ then it must also be a saturated subset of $Y.$ To see this, let $u,v \in Z$ such that $v$ covers $u$ in $Z$ and suppose $u \le y \le v$ for some $y \in Y.$ Then $y \in Z$ by definition, so $u = y$ or $v = y$ since $v$ covers $u$ in $Z.$ \end{remark}

\begin{lemma}\label{isomorphic}
If $f: X \to Y$ is a saturated embedding of posets, then $f$ is a poset isomorphism from $X$ onto $f(X),$ and $f(X)$ is a saturated subset of $Y.$
\end{lemma}

\begin{proof}
Since $f$ maps $X$ onto $f(X)$ and is assumed to be an embedding of posets, it follows that $f$ is an isomorphism from $X$ onto $f(X)$. Suppose $f(x),f(y) \in f(X)$ such that $f(y)$ covers $f(x)$ in $f(X)$. Let $z \in Y$ such that $f(x) \le z \le f(y)$.  Since $f$ is a poset isomorphism from $X$ onto $f(X)$, $y$ covers $x$ in $X$.  Since $f$ is a saturated embedding from $X$ to $Y$, we have that $f(y)$ covers $f(x)$ in $Y$.  It follows that $z = f(x)$ or $z = f(y).$ Therefore, $f(X)$ is a saturated subset of $Y.$  \end{proof}

In this article, we show that, given a finite poset $X$, there exists a Noetherian UFD $A$ and a saturated embedding of posets $\psi: X \longrightarrow \Spec(A)$.  Lemma \ref{isomorphic} then implies that $X$ is isomorphic to a saturated subset of the spectrum of a Noetherian UFD. Similarly, we show that, given a finite poset $X$, there exists a quasi-excellent domain $B$ and a saturated embedding of posets $\vp: X \longrightarrow \Spec(B)$.  Lemma \ref{isomorphic} then implies that $X$ is isomorphic to a saturated subset of the spectrum of a quasi-excellent domain.

\section{Growing and Gluing}\label{Growing and Gluing}

%NOTE:  The intro here needs to be reworked since, for example, we no longer talk about chain extensions.

\medskip

In this section, we show that it is possible to ``grow" out of a minimal node as well as ``glue" minimal nodes as discussed in the introduction.  We begin by describing the growing process.  Let $B$ be a local ring satisfying some mild desirable conditions.  For the first result in this section, we construct a local ring $S$ and ensure that $\Spec(S)$ and $\Spec(B)$ are related in a specific way.
%In this section, we assume we are given a local ring $B$ satisfying some mild conditions.  Our goal is to construct a local ring $S$ such that, if $\Spec(B)$ contains a particular finite poset $X$, then $\Spec(S)$ contains a ``grown" version of $X$, as described in the Introduction.  The main result in this section, which we call The Growing Theorem, shows that we can do this in the case that the ``grown" version of $X$ has dimension one greater than that of $X$.  In Sections \ref{Posets} and \ref{MainResult}, we make these ideas more precise and we use induction to show that we can generalize to the case where the ``grown" version of $X$ has dimension more than one greater than the dimension of $X$.
In particular, to construct $S$, we focus on the minimal prime ideals of $B$.  
%Assume that $X$ has $m$ minimal nodes and $B$ has $m$ minimal prime ideals corresponding to the minimal nodes of $X$. 
 The ring $S$ will be of the form $B[[y,z]]/J$ where $y$ and $z$ are indeterminates and $J$ is a carefully chosen ideal of $B[[y,z]]$.   In particular, if $\{P_1, P_2, \ldots ,P_m\}$ are the minimal prime ideals of $B$, we choose $J$ such that $S$ has a height one prime ideal that corresponds to $P_1$ and such that, for $j = 2,3, \ldots ,m$, $S$ has a minimal prime ideal that corresponds to $P_j$.  In addition, the height one prime ideal of $S$ corresponding to $P_1$ contains as many minimal prime ideals of $S$ as we desire.  The idea is that the minimal prime ideals of $S$ and $B$ are the same, except that, in $S$, we have ``grown" some minimal prime ideals out of $P_1$.  
 %We also ensure that the only minimal prime ideals of $S$ are the ones contained in , if $Q$ is a minimal prime ideal of $S$ and does not correspond to one of the $m - k$ minimal prime ideals of $B$, then $Q$ is contained in exactly one of the aforementioned $k$ height one prime ideals of $S$.  
 
 %Since we use this result to construct rings with increasing dimension and desirable spectra, we refer to this result as The Growing Theorem (Theorem \ref{onegrowing}).

In Section \ref{GluingNodes}, we define precisely what we mean when we ``grow" nodes out of one minimal node of a poset.  We first describe a process in which, given a finite poset $Z$, we retract height zero nodes into a chosen height one node of $Z$, and the new poset will be called a retraction of $Z$.  The idea of a retracting height zero nodes into a height one node is the reverse of the growing process.  So, if the poset $X$ is a retraction of the poset $Z$, then, informally, $Z$ can be obtained from $X$ by growing some nodes out of a minimal node of $X$. In Section \ref{MainResults}, we show that, if $X$ is a retraction of $Z$ and $\vp:X \longrightarrow \Spec(B)$ a saturated embedding, then there is a local ring $S$ satisfying the same desirable conditions as $B$ and such that there is a saturated embedding $\psi: Z \longrightarrow \Spec(S)$.  As a consequence, if $X$ is isomorphic to a saturated subset of $\Spec(B)$, then there is a local ring $S$ satisfying the same desirable properties that $B$ satisfies such that $Z$ is isomorphic to a saturated subset of $\Spec(S)$.  This will enable us to construct rings containing larger and larger parts of our original finite poset in their spectrum.  Since the next theorem is crucial for constructing these rings with larger parts of our poset in their spectrum, we refer to it as The Growing Theorem.

%In other words, going from $\Spec(B)$ to $\Spec(S)$, we drop as many nodes as we desire from $k$ minimal nodes of $\Spec(B)$.  

%We use this theorem to grow the posets as described in the Introduction, and thus we call our main result of this section The Growing Theorem.

%To prove The Growing Theorem, we begin with an elementary lemma.

%\begin{lemma}\label{minimalprimes}
%Let $B$ be a Noetherian ring and let $x_1,x_2, \ldots ,x_n$ with $n \geq 1$ be indeterminates.  Let $A = B[[x_1, x_2,  \ldots ,x_n]]$.  Suppose that $\Min (B) = \{P_1, P_2, \ldots ,P_r\}$.  Then $\Min (A) = \{P_1A, P_2A, \ldots ,P_rA\}$.
%\end{lemma}

%\begin{proof}
%Note that, since $A/P_iA \cong B/P_i[[x_1, \ldots x_n]]$, we have $P_iA \in \Spec(A)$ for every $i$.  

%Suppose $Q \in \Spec(A)$ with $Q \subseteq P_iA$ for some $i$.  Then $Q \cap B \subseteq P_iA \cap B = P_i$, and so $Q \cap B = P_i$.  It follows that $P_i A = (Q \cap B)A \subseteq Q$ and we have $Q = P_iA$.  Therefore, $P_iA$ is a minimal prime ideal of $A$ for every $i = 1,2, \ldots , r$.

%Now suppose $J \in \Min(A)$.  Then there is an $i$ such that $P_i \subseteq J \cap B$, and so $P_iA \subseteq (J \cap B)A \subseteq J$.  As $J$ is minimal, $P_i A = J$.
%\end{proof}

%PUT IN COMMENT ABOUT WHAT (P,y,z)A MEANS.

%We are now ready to prove The Growing Theorem.

\begin{theorem} (The Growing Theorem) \label{onegrowing}
Let $B$ be a local ring containing an infinite field $K$.
%and let $T= \widehat{B}$ be the completion of $B$ at its maximal ideal.  
%Suppose that, if $I \in \Min (T)$, then $T/I$ is a RLR.
Let $\Min(B) = \{P_1, P_2 \ldots ,P_m\}$.  Let $n$ be a positive integer, let $y$ and $z$ be indeterminates, and let $A = B[[y,z]]$.  Then there are distinct prime ideals $Q_{1}, Q_{2} \ldots ,Q_{n}$ of $A$ such that

%\medskip
\vspace{.1cm}

\begin{enumerate}
%\item ht$(Q_{ij}) = 1$ for every $i = 1,2, \ldots ,k$ and for every $j = 1,2, \ldots ,n_i$. \\
\item For every $i = 1,2, \ldots ,n$, we have $Q_{i} \subseteq (P_{j}, y, z)A$ if and only if $j = 1$. 
\item For every $i = 1,2, \ldots ,n$, the prime ideal $(P_{1}, y, z)A/Q_{i}$ in the ring $A/Q_{i}$ has height one, and 
\item If $J = (\cap_{i=1}^{n}Q_{i}) \cap P'_{2} \cap \cdots \cap P'_m$ where $P_j' = (P_j, y,z)A$ then the minimal prime ideals of $S = A/J$ are $\{Q_1/J, Q_2/J, \ldots ,Q_n/J, P'_2/J, \ldots P'_m/J\}$, and $(P_1,y,z)/J$ is a height one prime ideal of $S$ containing exactly $n$ minimal prime ideals of $S$, namely $\{Q_1/J, Q_2/J, \ldots ,Q_n/J\}$.
%\item If $J = (\cap_{i=1}^{n}Q_{i}) \cap P'_{2} \cap \cdots \cap P'_m$ where $P_j' = (P_j, y,z)A$ and $S = A/J$, then, if $Q$ is a minimal prime ideal of $\widehat{S}$, we have that $\widehat{S}/Q$ is a RLR.  
%\item If $\Phi: F \longrightarrow \Spec(B)$ is a saturated embedding, then $\Psi: F \longrightarrow \Spec(S)$ given by $\Psi(u) = (\Phi(u)A + yA + zA)/J$ is a saturated embedding.  And I think we need something here about ``growing" the poset so that there is a saturated embedding from a larger part of the ``desired poset" to $\Spec(S)$.
\end{enumerate}
\end{theorem}

%\vspace{.05cm}

\begin{proof}
Since $B$ contains $K$, so does $A$.  Let $\beta_{i}$ for $i = 1,2, \ldots ,n$ be distinct elements of $K$.  For $i = 1,2, \ldots ,n$, define $Q_{i} = (P_1, y + \beta_{i}z)A$.  Then, $A/Q_{i} \cong B/P_1[[z]]$, and so $Q_{i}$ is a prime ideal of $A$.  If $Q_{i} = Q_{\ell}$ for some $i$ and for some $\ell$, then, in $A/P_1A \cong B/P_1[[y,z]]$, we have $(y + \beta_{i}z) = (y + \beta_{\ell}z)$.  Hence $y + \beta_{i}z - y - \beta_{\ell}z = z(\beta_{i} - \beta_{\ell}) \in (y + \beta_{i}z)$.  If $i \neq \ell,$ then we have $z,y \in (y + \beta_{i}z)$, a contradiction.
%So $y + \beta_{ij}z = \alpha y + \alpha \beta_{i\ell}z$ for some unit $\alpha$.  Hence, $y(1 - \alpha) = z(\alpha \beta_{i\ell} - \beta_{ij})$.  For the degree one terms to be equal on both sides of this equation, the constant term of $\alpha$ must be $1$.  So, $\alpha = 1 + m$ where $m \in (y,z)$.  So $y(-m) = z(m\beta_{i\ell} + \beta_{i\ell} - \beta_{ij})$.  Since the degree one term on the left hand side is zero, the degree one term on the right hand side is also zero, 
%So $\beta_{i\ell} = \beta_{ij},$.  
It follows that $i = \ell$.  
%If $i \neq r$, then there is an element $b \in P_i$ with $b \not\in P_r$, and so $Q_{ij} \neq Q_{rs}$.  
Therefore, the $Q_{i}$'s are distinct.

%By Lemma \ref{minimalprimes}, the minimal prime ideals of $A$ are $\{P_1A, \ldots ,P_kA, P_{k + 1}A, \ldots ,P_mA\}$.  Note that the only minimal prime ideal of $A$ contained in $Q_{ij}$ is $P_iA$.  In the ring $A/P_iA \cong B/P_i[[y,z]]$, the prime ideal $(y + \beta_{ij}z)$ has height one, and so ht$Q_{ij} = 1$ in the ring $A$.

By the definition of $Q_{i}$, we have that $Q_{i} \subseteq (P_{1}, y, z)A$.  If $j \neq 1$, then there is a $\gamma \in P_j$ with $\gamma \not\in P_{1}$, and so $Q_{i} \not\subseteq (P_{j}, y, z)A$.  Since the prime ideal $(P_{1}, y, z)A/Q_{i}$ in the domain $A/Q_{i}$ is nonzero and principal, it has height one.
%Suppose $Q$ is a minimal prime ideal of $\widehat{S}$.  Then $Q \cap S$ is either $Q_{i}/J$ for some $i = 1,2, \ldots ,n$ or $Q \cap S$ is in $\{P'_{2}/J, \ldots , P'_m/J\}$.  If $Q \cap S = Q_{i}/J$, then $S/(Q \cap S) \cong B/P_1[[z]]$, and so the completion of $S/(Q \cap S)$ is isomorphic to $T/P_1T[[z]]$.  By Lemma \ref{minimalprimes}, the minimal prime ideals of $T/P_1T[[z]]$ are of the form $I/P_1T[[z]]$ where $I \in \Min(T)$.  The ring $T/P_1T[[z]]$ modulo $I/P_1T[[z]]$ is isomorphic to $T/I[[z]]$, which is a RLR since $T/I$ is a RLR.  It follows that the completion of $S/(Q \cap S)$ modulo any of its minimal prime ideals is a RLR.  Noting that $Q/(Q \cap S)\widehat{S}$ is a minimal prime ideal of $\widehat{S}/(Q \cap S)\widehat{S}$, we have that $\widehat{S}/Q$ is a RLR.  If $Q \cap S = P'_j/J$ for some $j = 2, \ldots m$, then $S/(Q \cap S) \cong B/P_j$, and so the completion of $S/(Q \cap S)$ is isomorphic to $T/P_jT$. If $I/P_jT$ is a minimal prime ideal of $T/P_jT$, then $T/P_jT$ modulo $I/P_jT$ is isomorphic to $T/I$ where $I$ is a minimal prime ideal of $T$.  By hypothesis, this is a RLR, and so the completion of $S/(Q \cap S)$ modulo any of its minimal prime ideals is a RLR.  Since $Q/(Q \cap S)\widehat{S}$ is a minimal prime ideal of $\widehat{S}/(Q \cap S)\widehat{S}$, we have that $\widehat{S}/Q$ is a RLR. 

The last condition follows by the first two conditions and the definition of $S$.
\end{proof}

%Our goal is to apply Theorem \ref{gluing} repeatedly to ``glue" together the minimal prime ideals of $B$ in any prescribed way while not altering the spectrum of $B$ at its nonminimal prime ideals.  To do this, we first show that the rings $S$ and $B$ in Theorem \ref{gluing} are related in several nice ways.  In particular, we show they have the same completion, and that the sets $\{P \in \Spec(B) \, | \, P \neq Q_1 \mbox{ and } P \neq Q_2\}$ and $\{P \in \Spec(S) \, | \, P \neq Q_1 \cap S\}$ are isomorphic as posets.  To show that they have the same completion we use the following result.

%We are now ready to state and prove The Gluing Theorem.  In our proof, we use Theorem \ref{gluing} and induct on the number of minimal prime ideals of $B$.  We also use the following observation.  Suppose $S$ is a subring of the ring $B$ and $P$ is a prime ideal of $B$ with positive height satisfying $(S \cap P)B = P$.  Then $S \cap P$ is not a minimal prime ideal of $S$.  To see this, observe that since $P$ has positive height, it strictly contains a minimal prime ideal $Q$ of $B$.  If $S \cap P$ is a minimal prime ideal of $S$, then $S \cap P = S \cap Q$ and so $P = (S \cap P)B = (S \cap Q)B \subseteq Q$, a contradiction.

The next result is taken from \cite{GluingPaper}.  Suppose $(S,M)$ is a reduced local ring satisfying some mild conditions.  Let $\Min(S)$ be partitioned into $m \geq 1$ subcollections $C_1, \ldots ,C_m$.  Theorem \ref{biggluing} shows that there exists a subring $(B, B \cap M)$ of $S$ such that $S$ satisfies the same mild conditions as $B$ and such that $B$ has exactly $m$ minimal prime ideals.  Specifically, if $Q,Q'$ are minimal prime ideals of $S$ then $Q \cap B = Q' \cap B$ if and only if $Q,Q' \in C_i$ for some $i = 1,2, \ldots ,m$. In addition, the spectrum of $B$ when restricted to the prime ideals of positive height is isomorphic to the spectrum of $S$ when restricted to the prime ideals of positive height.  So we can think of the spectrum of $B$ to be the same as the spectrum of $S$ except that, to get the minimal prime ideals of $B$, the minimal prime ideals of $S$ are ``glued" together based on the partition of $\Min(S)$.  It is worth noting that there are no restrictions on the partition.  That is, one can glue together the minimal primes in any prescribed way.  For this reason, we refer to Theorem \ref{biggluing} as The Gluing Theorem.

\medskip

\begin{theorem}[\cite{GluingPaper}, Theorem 2.14](The Gluing Theorem)\label{biggluing}
Let $(S,M)$ be a reduced local ring containing the rationals with $S/M$ uncountable and $|S| = |S/M|$.  
%Suppose $B$ contains the rationals and $\Min(B) = \{Q_1, Q_2, \ldots ,Q_n\}$ with $n \geq 2$, and 
Suppose $\Min(S)$ is partitioned into $m \geq 1$ subcollections $C_1, \ldots ,C_m$. Then there is a reduced local ring $B \subseteq S$ with maximal ideal $B \cap M$ such that 
\medskip
\begin{enumerate}
\item $B$ contains the rationals, 
\item $\widehat{B} = \widehat{S}$, 
\item $B/(B \cap M)$ is uncountable and $|B| = |B/(B \cap M)|$, 
\item For $Q, Q' \in \Min(S)$, $Q \cap B = Q' \cap B$ if and only if there is an $i \in \{1,2, \ldots ,m\}$ with $Q \in C_i$ and $Q' \in C_i$, 
\item The map $f:\Spec(S) \longrightarrow \Spec(B)$ given by $f(P) = B \cap P$ is onto and, if $P$ is a prime ideal of $S$ with positive height, then $f(P)S = P$.  In particular, if $P$ and $P'$ are prime ideals of $S$ with positive height, then $f(P)$ has positive height and $f(P) = f(P')$ implies that $P = P'$. 
%\item $\{P \in \Spec(B) \, | \, P \not\in \Min(B)\}$ and $\{P \in \Spec(S) \, | \, P \not\in \Min(S)\}$ are isomorphic as posets.
\end{enumerate}
\end{theorem}

Once we have a local ring $B$ and a saturated embedding from our poset $X$ to $\Spec(B)$, we use $B$ to construct a UFD for which there is also a saturated embedding from $X$ to the spectrum of the UFD.  The elements of the poset $X$ will ultimately correspond to prime ideals of the UFD that have small coheight.  To construct our UFD, we choose $n$ large enough so that $B[[x_1, \ldots ,x_n]]$ satisfies the property that all associated prime ideals of all principal ideals have large coheight. Theorem \ref{growcoheight} ensures we can do this. This property allows us to construct a subring of $B[[x_1, \ldots ,x_n]]$ that is a local UFD.  In addition, we are able to ensure that the part of the spectrum of the UFD of small coheight and the part of the the spectrum of $B[[x_1, \ldots ,x_n]]$ of small coheight are isomorphic as posets.  From this, we conclude in Section \ref{MainResults} that there is a saturated embedding from $X$ to the spectrum of the UFD.

\begin{theorem}\label{growcoheight}
Let $(B',M')$ be a local ring and let $x_1, x_2, \ldots ,x_n$ be indeterminates where $n \geq 1$.  If $f \in B'[[x_1, x_2, \ldots ,x_n]]$ is a regular element of $B'[[x_1, x_2, \ldots ,x_n]]$ and $Q \in \Spec (B'[[x_1, x_2, \ldots ,x_n]])$ with $Q \in \Ass (B'[[x_1, x_2, \ldots ,x_n]]/fB'[[x_1, x_2, \ldots ,x_n]])$ then dim$(B'[[x_1, x_2, \ldots ,x_n]]/Q) \geq n - 1$.
\end{theorem}

\begin{proof}
Since $f$ is regular, depth$(B'[[x_1, x_2, \ldots ,x_n]]/fB'[[x_1, x_2, \ldots ,x_n]]) = \mbox{depth}(B'[[x_1, x_2, \ldots ,x_n]]) - 1 \geq n - 1$. Applying Chapter 6, Theorem 29 from \cite{matsumuracommalg} to the ring $$A = (B'[[x_1, x_2, \ldots ,x_n]]/fB'[[x_1, x_2, \ldots ,x_n]]),$$ we have $n - 1 \leq \mbox{depth}(A) \leq \mbox{dim}(A/Q) = \mbox{dim}(B'[[x_1, x_2, \ldots ,x_n]]/Q)$.
\end{proof}

We now turn our attention to embedding our poset $X$ into the spectrum of a quasi-excellent domain. Recall that a Noetherian ring $R$ is said to be quasi-excellent if it satisfies the following two conditions:

\begin{enumerate}
    \item for all $P \in \Spec{(R)}$, the ring $\widehat{R} \otimes_R L$ is regular for every finite field extension $L$ of $R_P/PR_P$;
    \item Reg$(B) \subset \Spec(B)$ is open for every finitely generated $R$-algebra $B$.
\end{enumerate}

\begin{remark}\label{quasiexc}
Suppose that $R$ is a semi-local ring. Then by Chapter 13, Theorem 76 and Lemma 33.4 in \cite{matsumuracommalg}, if $R$ satisfies condition (1), it also satisfies condition (2).  Thus, when a ring $R$ is semi-local, to show that it is quasi-excellent, it suffices to verify that condition (1) holds.
\end{remark}

The next result shows that if the ring $B$ in Theorem \ref{onegrowing} is quasi-excellent, then the ring $S$ constructed in Theorem \ref{onegrowing} is quasi-excellent.

\begin{proposition} \label{growingexcellent}
Suppose the ring $B$ in Theorem \ref{onegrowing} satisfies the additional condition that it is quasi-excellent. Then the ring $S$ constructed in Theorem \ref{onegrowing} is also quasi-excellent.
\end{proposition}

\begin{proof}
It is shown in \cite{RotthausGerman} that if $B$ is a quasi-excellent local ring, then so is $B[[y]]$.  It follows that $A$ in Theorem \ref{onegrowing} is quasi-excellent.  Since a quotient ring of a quasi-excellent ring is also quasi-excellent, the result follows.
\end{proof}

We now argue that if the ring $S$ in Theorem \ref{biggluing} is quasi-excellent, then the ring $B$ constructed in Theorem \ref{biggluing} is quasi-excellent. We first note that, to show that condition (1) in the definition of quasi-excellent holds, we can restrict to purely inseparable finite field extensions (see, for example, Remark 1.3 in \cite{Rotthaus}). Hence, if a local ring $R$ contains the rationals, then $R$ satisfies condition (1) in the definition for quasi-excellent if, for all $P \in \Spec{(R)}$, the ring $\widehat{R} \otimes_R (R_P/PR_P)$ is a regular ring.
Thus, to show that $R$ satisfies condition (1) in the definition of quasi-excellent, it suffices to show that, for all $P \in \Spec{(R)}$ and for all $Q \in \Spec{(\widehat{R})}$ satisfying $Q \cap R = P$, we have that $(\widehat{R}/P\widehat{R})_Q$ is a RLR.

\begin{proposition}\label{excellentgluing}
Suppose the ring $S$ in Theorem \ref{biggluing} satisfies the additional condition that it is quasi-excellent. Then the ring $B$ constructed in Theorem \ref{biggluing} is also quasi-excellent.
\end{proposition}

\begin{proof}
Let $T = \widehat{B} = \widehat{S}$. We show that if $P \in \Spec{(B)}$ and $\widehat{P} \in \Spec{(T)}$ such that $\widehat{P} \cap B = P$, then $(T/PT)_{\widehat{P}}$ is a RLR.

First, suppose $P \in \Spec{(B)}$ such that $P$ has positive height, and let $\widehat{P} \in \Spec{(T)}$ such that $\widehat{P} \cap B = P$.  Let $P_S = S \cap \widehat{P}$.  Then $P_S \cap B = P$.  Since $P$ has positive height, we have, by construction, that $P_S$ has positive height and $PS = P_S$.  It follows that $PT = P_ST$ and so $(T/PT)_{\widehat{P}} = (T/P_ST)_{\widehat{P}}$.  Since $S$ is quasi-excellent, $(T/P_ST)_{\widehat{P}}$ is a RLR and so $(T/PT)_{\widehat{P}}$ is as well.

Now suppose $P$ is a minimal prime ideal of $B$ and let $\widehat{P} \in \Spec{(T)}$ such that $\widehat{P} \cap B = P$.  Let $P_S = S \cap \widehat{P}$. If $P_S$ has positive height, then, by construction, $B \cap P_S = P$ has positive height, a contradiction.  Hence, $P_S$ is a minimal prime ideal of $S$.  If $Q$ is a minimal prime ideal of $T$ contained in $\widehat{P}$, then $Q \cap S \subseteq \widehat{P} \cap S = P_S$, and so $Q \cap S = P_S$.  In particular, $P_ST \subseteq Q$ for all minimal prime ideals $Q$ of $T$ that are contained in $\widehat{P}$.  Since $S$ is reduced and quasi-excellent, $\widehat{S} = T$ is reduced.  Thus, $T_{\widehat{P}}$ is reduced and so the intersection of its minimal prime ideals is the zero ideal.  Therefore, $P_ST_{\widehat{P}}$ is the zero ideal of $T_{\widehat{P}}$.  As $S$ is quasi-excellent, $(T/P_ST)_{\widehat{P}} \cong T_{\widehat{P}}$ is a RLR.  Noting that $PT \subseteq P_ST$, we have that $PT_{\widehat{P}}$ is also the zero ideal of $T_{\widehat{P}}$ and so $(T/PT)_{\widehat{P}} \cong T_{\widehat{P}}$ is a RLR.
It follows that $B$ satisfies condition (1) in the definition of quasi-excellent.  Since $B$ is a local ring, it is quasi-excellent by Remark \ref{quasiexc}.
\end{proof}

Let $X$ be a finite poset.  In Section \ref{GluingNodes}, we define precisely what we mean for $Y$ to be a poset obtained from $X$ by gluing some minimal nodes together.  The resulting poset $Y$ will be called a height zero gluing of $X$.  Suppose that $Z$ is a nonempty finite poset, $Y$ is a height zero gluing of $Z$,  $(S,M)$ is a local ring satisfying some mild conditions, and $\psi: Z \rightarrow \Spec(S)$ is a saturated embedding.  In Section \ref{MainResults}, we use Theorem \ref{biggluing} and Proposition \ref{excellentgluing} to show that there is a subring $(B,B \cap M)$ of $S$ such that $B$ satisfies the same mild conditions as $S$ and such that there is a saturated embedding $\vp: Y \rightarrow \Spec(B)$.  In other words, we show that if $Y$ is a height zero gluing of a poset that is isomorphic to a saturated subset of the spectrum of a ring satisfying our nice properties, then $Y$ is also isomorphic to a saturated subset of the spectrum of a ring satisfying the same nice properties.  This result enables us to construct nice rings that contain height zero gluings of our posets in their spectra.

%\newpage

\section{The UFD Theorem}\label{UFDTheorem}

The last step of our procedure is to construct a Noetherian UFD whose spectrum contains a saturated subset that is isomorphic to the given finite poset.  In this section, we describe the construction of this UFD.  We start with a local ring $B$ that contains the rationals and that satisfies other mild conditions. The spectrum of the UFD that we construct will inherit important properties from the spectrum of $B$.  In particular, given $0 \leq h \leq \mbox{dim}B - 2$, the set of prime ideals of our UFD of coheight at most $h$ and the set of prime ideals of $B$ of coheight at most $h$ will be isomorphic when viewed as posets.  The construction is similar to the procedure used in \cite{GluingPaper} to prove The Gluing Theorem (Theorem \ref{biggluing}).  We start with the rationals and we successively adjoin uncountably many elements of $B$ to form an increasing chain of subrings of $B$ that satisfy desirable properties, the most important property being that they are all UFD's.  Our final UFD will be the union of this increasing chain.  We use Proposition \ref{completionsame} below to show that it is Noetherian and its completion is the same as that of $B$.  We also ensure that the UFD contains generating sets for carefully chosen prime ideals of $B$. This property will ultimately help us show that the spectrum of $B$ and the spectrum of the UFD, when restricted to prime ideals of coheight at most $h$, are isomorphic.

\begin{proposition}[\cite{GluingPaper}, Proposition 2.6]\label{completionsame}
Let $(B,M)$ be a local ring and let $T = \widehat{B}$. Suppose $(S,S \cap M)$ is a quasi-local subring of $B$ such that the map $S \longrightarrow B/M^2$ is onto and $IB \cap S = I$ for every finitely generated ideal $I$ of $S$.  Then $S$ is Noetherian and $\widehat{S} = T$.  Moreover, if $B/M$ is uncountable and $|B| = |B/M|$, then $S/(S \cap M)$ is uncountable and $|S| = |S/(S \cap M)|$.
\end{proposition}

Most of the material in this section is based on the ideas in \cite{heitmannUFD}.  The following definition is taken from \cite{heitmannUFD}.

%\begin{lemma}\label{MinLemma}
%Let $B'$ be a local ring with completion $T'$.  If $P' \in \Spec(B')$ with $P' \in \Ass(B')$, then $P'T' \subseteq Q'$ for some $Q' \in \Ass(T')$.  
%\end{lemma}

%\begin{proof}
%Since $P' \in \Ass(B')$, we have that $P' = \mbox{Ann}(x)$ for some nonzero $x$ in $B'$.  Let $P' = (p_1, \ldots ,p_k)$, and let $y \in P'T'$.  Then $y = p_1t_1 + \cdots + p_kt_k$ for some $t_i \in T'$, and $xy = xp_1t_1 + \cdots + xp_kt_k = 0.$  So $y = 0$ or $y$ is a zero divisor of $T'$.  Hence, $P'T'$ consists only of zero divisors.  By the Prime Avoidance Theorem, $P'T' \subseteq Q'$ for some $Q' \in \Ass(T')$.
%\end{proof}

%\medskip

%\noindent NOTE:  For the paper, make the above shorter -- maybe get rid of the first lemma entirely.

%\medskip

\begin{definition}
Let $(T,M)$ be a complete local ring and let $(R, M \cap R)$ be a quasi-local unique factorization domain contained in $T$ satisfying:
\begin{enumerate}
\item $|R| \leq \sup(\aleph_0, |T/M|)$ with equality only if $T/M$ is countable,
\item $Q \cap R = (0)$ for all $Q \in \Ass(T)$, and
\item if $t \in T$ is regular and $P \in \Ass(T/tT)$, then ht$(P \cap R) \leq 1$.
\end{enumerate}
Then $R$ is called an \textit{N-subring of} $T$.
\end{definition}

We modify the above definition for our purposes.  In particular, we construct subrings of the ring $B$, which is not necessarily complete.  So, instead of requiring our rings to be N-subrings, we require them to be PN-subrings as defined below.

\begin{definition}
Let $(B,M)$ be a local domain with $B/M$ uncountable, and let $(R, M \cap R)$ be an infinite quasi-local unique factorization domain contained in $B$ such that $|R| < |B/M|$ and, if $b \in B$ and $P \in \Ass(B/bB)$, then ht$(P \cap R) \leq 1$.  Then $R$ is called a \textit{Pseudo-N-Subring of} $B$, or a \textit{PN-subring of} $B$.
\end{definition}

Note that, if $B$ in the above definition is complete and if $(R,M \cap R)$ is a PN-subring of $B$, then $(R,M \cap R)$ is also an N-subring of $B$.

The next result describes a sufficient condition on an element $x$ of $B$ so that, if $R$ is a PN-subring of $B$ then the ring $R[x]_{(R[x] \cap M)}$ is also a PN-subring of $B$.  This will allow us to successively adjoin elements while maintaining the PN-subring properties.  Note that if $P$ is a prime ideal of $B$, then there is an injective homomorphism $R/(R \cap P) \longrightarrow B/P$ and so we can view $R/(R \cap P)$ as a subring of $B/P$.

\begin{lemma}\label{adjoin}
Let $(B,M)$ be a local domain with $B/M$ uncountable, and let $(R, M \cap R)$ be a PN-subring of $B$.  Let $C = \{P \in \Spec(B) \, | \, P \in \Ass(B/rB) \mbox{ for some } r \in R\}$.  Let $x \in B$ be such that $x + P \in B/P$ is transcendental over $R/(R \cap P)$ for every $P \in C$.  Then $S = R[x]_{(R[x] \cap M)}$ is a PN-subring of $B$ with $|S| = |R|$. Moreover, prime elements in $R$ are prime in $S$.  
\end{lemma}

\begin{proof}
Since $R$ is infinite and $x$ is transcendental over $R$, we have that $|S| = |R|$, $S$ is a UFD, and prime elements in $R$ are prime in $S$.  Now let $b \in B$ and $P \in \Ass(B/bB)$.  Since $R$ is a PN-subring, ht$(R \cap P) \leq 1$.  If $R \cap P = (0)$, then in the ring $R[x]_{(R[x] \cap P)}$, all nonzero elements of $R$ are units, and so the ring $R[x]_{(R[x] \cap P)}$ is isomorphic to a ring $k[x]$ with some elements inverted where $k$ is a field.  It follows that the ring $R[x]_{(R[x] \cap P)}$ has Krull dimension at most 1, and so ht$(R[x] \cap P) \leq 1$.  Therefore, ht$(S \cap P) \leq 1$.  Now suppose ht$(R \cap P) = 1$.  Then $R \cap P = zR$ for some nonzero $z$ in $R$.  Since $P \in \Ass(B/bB)$, we have $PB_P \in \Ass(B_P/bB_P)$.  It follows that depth$B_P = 1$ and so $PB_P \in \Ass(B_P/zB_P)$.  Hence $P \in Ass(B/zB)$ and therefore $P \in C$.  Let $f \in R[x] \cap P$.  Then $f = r_nx^n + \cdots + r_1x + r_0$ for $r_i \in R$.  Since $x + P \in B/P$ is transcendental over $R/(R \cap P)$, we have that $r _i \in R \cap P = zR$ for every $i = 1,2 \ldots ,n$.  Hence, $f \in zR[x]$, and it follows that $(0) \neq R[x] \cap P \subseteq zR[x]$. Note that, since $z$ is a prime element of $R$, it is also a prime element of $R[x]$.  Hence, $zR[x]$ is a height one prime ideal of $R[x]$ and so $R[x] \cap P = zR[x]$. Therefore,
%We claim that ht$(zR[x]) = 1$.  To see this, suppose $Q$ is a prime ideal of $R[x]$ with $(0) \subseteq Q \subseteq zR[x]$ and $Q \neq zR[x]$.  Let $q \in Q$.  Then $q = zg_1$ for some $g_1 \in R[x]$.  Since $z \not\in Q$, we have $g_1 \in Q$.  So $g_1 = zg_2$ for some $g_2 \in R[x]$.  Note that $g_2 \in Q$.  Continuing this process, we have that $q \in \cap_{i = 1}^{\infty} z^iB$.  Since $B$ is Noetherian, $\cap_{i = 1}^{\infty} z^iB = (0)$, and it follows that $Q = (0)$.  Hence, ht$(zR[x]) = 1$ and so 
ht$(R[x] \cap P) = 1$ and it follows that ht$(S \cap P) = 1$. Thus, $S$ is a PN-subring satisfying the desired properties.
\end{proof}

For the use of Proposition \ref{completionsame}, we guarantee that our final UFD contains an element of every coset of $B/M^2$.  We first state a result from \cite{heitmannUFD} that can be thought of as a generalization of the prime avoidance theorem.  Then, in Lemma \ref{coset2}, we use this result to show that, given a PN-subring $R$ of $B$, we can adjoin an element of a coset of $B/M^2$ to $R$ while preserving many desirable properties.

\begin{lemma}[\cite{heitmannUFD}, Lemma 3] \label{primeavoid}
%\label{superprimeavoidance}
Let $(B,M)$ be a local ring.  Let $C \subseteq \Spec(B)$, let $I$ be an ideal of $B$ such that $I \not\subseteq P$ for every $P \in C$, and let $D$ be a subset of $B$.  Suppose $|C \times D| < |B/M|$.  Then $I \not\subseteq \bigcup \{(P + r) \, | \, P \in C, r \in D\}.$
\end{lemma}

\begin{lemma}\label{coset2}
Let $(B,M)$ be a local domain with $B/M$ uncountable and depth$B \geq 2$, and let $(R, M \cap R)$ be a PN-subring of $B$.  Let $u \in B$.  Then there exists a PN-subring $(S, M \cap S)$ of $B$ with $R \subseteq S$, $|S| = |R|$, prime elements in $R$ are prime in $S$, and $S$ contains an element of the coset $u + M^2$.
\end{lemma}

\begin{proof}
For $P \in \Spec(B)$, define $D_{(P)}$ to be a full set of coset representatives of the cosets $t + P \in B/P$ that make $(t + u) + P$ algebraic over $R/(P \cap R)$.  Note that, since $R$ is infinite, $|D_{(P)}| \leq |R|$.  Let $C = \{P \in \Spec(B) \, | \, P \in \Ass(B/rB) \mbox{ for some } r \in R\}$, and let $D = \bigcup_{P \in C}D_{(P)}$.  Since depth$B > 1$, $M \not\subseteq P$ for every $P \in C$, and so $M^2 \not\subseteq P$ for every $P \in C$.  Note that $|C \times D| \leq |R| < |B/M|$.  By Lemma \ref{primeavoid}, there is an $m \in M^2$ such that $m \not\in \bigcup \{(P + r) \, | \, P \in C, r \in D\}$. Hence, $m + u + P$ is transcendental over $R/(P \cap R)$ for every $P \in C$.  By Lemma \ref{adjoin}, $S = R[m + u]_{(R[m + u] \cap M)}$ is a PN-subring of $B$ with $|S| = |R|$ and prime elements in $R$ are prime in $S$.  Finally, note that $S$ contains $m + u$, an element of the coset $u + M^2$.
\end{proof}

To obtain the desired relationship between the spectra of $B$ and our final UFD, we ensure that the UFD contains generating sets for certain prime ideals of $B$.  The next lemma will help accomplish this. 

%\begin{lemma}\label{heightone}
%Let $(B,M)$ be a local domain with $T = \widehat{B}$.  Suppose that if $J \in \Min(T)$ then $T/J$ is a RLR.  Let $r \in B$ with $r \neq 0$ and suppose that $P \in \Spec(B)$ with $P \in \Ass(B/rB)$.  Then ht$P = 1$.
%\end{lemma}

%\begin{proof}
%Since the completion of $B/rB$ is $T/rT$,
%Let $P' = P/rB$.  
%By Lemma \ref{MinLemma}, $P'T' \subseteq Q'$ for some $Q' \in \Spec(T')$ with $Q' \in \Ass(T')$.  
%It follows that 
%there is a $Q \in \Spec(T)$ such that 
%$r \in Q$, 
%$Q \in \Ass(T/rT)$, and $PT \subseteq Q$.  Let $J$ be a minimal prime ideal of $T$ contained in $Q$.  Then $Q/J \in \Ass(\frac{T/J}{r(T/J)})$.  Now, $T/J$ is a RLR and so it satisfies Serre's $(S_2)$ condition.  It follows that ht$(Q/J) \leq 1$ in $T/J$.  Since this holds for every minimal prime ideal $J$ contained in $Q$, we have ht$Q \leq 1$, and so ht$P = \mbox{ht}(PT \cap B) \leq \mbox{ht}(Q \cap B) \leq 1$.  Since $r \in P$ with $r \neq 0$, we have ht$P = 1$.
%\end{proof}

\begin{lemma}\label{generators2}
Let $(B,M)$ be a local domain with depth$B \geq 2$ and with $B/M$ uncountable and let $(R, M \cap R)$ be a PN-subring of $B$. 
%Let $T = \widehat{B}$ and suppose that if $J \in \Min(T)$ then $T/J$ is a RLR.
Suppose $Q$ is a prime ideal of $B$ such that if $P \in \Spec(B)$ with $P \in \Ass(B/bB)$ for some $b \in B$, then $Q \not\subseteq P$.  Then there exists a PN-subring $(S,M \cap S)$ of $B$ such that $|S| = |R|$, prime elements in $R$ are prime in $S$, and $S$ contains a generating set for $Q$.
\end{lemma}

\begin{proof}
Let $Q = (x_1, x_2, \ldots ,x_k)$, and let $C = \{P \in \Spec(B) \, | \, P \in \Ass(B/rB) \mbox{ for some } r \in R\}$.  By hypothesis,
%Lemma \ref{heightone}, if $P \in C$, then ht$P \leq 1$ and so 
if $P \in C$ then $Q \not\subseteq P$.  Now use Lemma \ref{primeavoid} with $D = \{0\}$ to find $z_1 \in Q$ such that $z_1 \not\in P$ for every $P \in C$.  Note that, since depth$B \geq 2$, we have $M \not\subseteq P$ for every $P \in C$.  Suppose $P \in C$ and $x_1 + tz_1 + P = x_1 + t'z_1 + P$ with $t,t' \in B$.  Then $z_1(t-t') \in P$ and since $z_1 \not\in P$, we have $t + P = t' + P$.  It follows that $x_1 + tz_1 + P = x_1 + t'z_1 + P$ if and only if $t + P = t' + P$.  Let $D_{(P)}$ be a full set of coset representatives for the cosets $t + P \in B/P$ that make $x_1 + z_1t + P$ algebraic over $R/(R \cap P)$.  Note that $|D_{(P)}| \leq |R| < |B/M|$.  Now use Lemma \ref{primeavoid} with $D = \bigcup_{P \in C} D_{(P)}$ to find $m_1 \in M$ such that $m_1 \not\in \bigcup \{(P + r) \, | \, P \in C, r \in D\}$. Then $x_1 + z_1m_1 + P$ is transcendental over $R/(R \cap P)$ for every $P \in C$.  By Lemma \ref{adjoin}, $R_1 = R[x_1 + z_1m_1]_{(R[x_1 + zm_1] \cap M)}$ is a PN-subring of $B$ with $|R_1| = |R|$ and prime elements in $R$ are prime in $R_1$.  Note that $(x_1 + z_1m_1, x_2, \ldots ,x_k) + MQ = Q$ and so by Nakayama's Lemma, $(x_1 + z_1m_1, x_2, \ldots ,x_k) = Q$.  

Repeat this procedure with $R$ replaced by $R_1$ and $x_1$ replaced by $x_2$ to find $z_2 \in Q$ and $m_2 \in M$ such that $R_2 = R_1[x_2 + z_2m_2]_{(R[x_2 + z_2m_2] \cap M)}$ is a PN-subring of $B$ with $|R_2| = |R_1|$, prime elements in $R_1$ are prime in $R_2$, and $(x_1 + z_1m_1, x_2 + z_2m_2, x_3, \ldots ,x_k) = Q$.

Continue this process to find PN-subrings $R_3, \ldots ,R_k$ so that $R_k$ is a PN-subring of $B$ satisfying $R \subseteq R_k$, $|R_k| = |R|$, prime elements in $R$ are prime in $R_k$, $Q = (x_1 + z_1m_1, x_2 + z_2m_2, \ldots ,x_k + z_km_k)$, and $x_j + z_jm_j \in R_k$ for all $j = 1,2, \ldots ,k$.  Then $S = R_k$ is the desired PN-subring of $B$.
\end{proof}

Recall that to use Proposition \ref{completionsame}, we need our subring $A$ of $B$ to satisfy the condition that, if $I$ is a finitely generated ideal of $A$, then $IB \cap A = I$.  We use Lemma \ref{closing} repeatedly to do this.

\begin{lemma}\label{closing}
Let $(B,M)$ be a local domain with $B/M$ uncountable, and let $(R, M \cap R)$ be a PN-subring of $B$.  Suppose $I$ is a finitely generated ideal of $R$ and $c \in R$ with $c \in IB$.  Then there exists a PN-subring $S$ of $B$ with $R \subseteq S$, $|S| = |R|$, prime elements in $R$ are prime in $S$, and $c \in IS$.
\end{lemma}

\begin{proof}
Lemma 4 in \cite{heitmannUFD} is the analogous statement of this result for N-subrings.  For the proof of Lemma 4, the author only uses that the ring is complete in the case that the residue field is countable.  Since we are assuming that $B/M$ is uncountable, the proof of Lemma 4 in \cite{heitmannUFD} works for this result.  
\end{proof}

The next lemma gives sufficient conditions on an ascending chain of PN-subrings of $B$ to ensure that the union is also a PN-subring of $B$.  
Before we state and prove the lemma, we provide a technical definition.

\begin{definition}
Let $\Psi$ be a well-ordered set and let $\alpha \in \Psi$.  Define $\gamma (\alpha) = \sup\{\beta \in \Psi \, | \, \beta < \alpha\}$.
\end{definition}

\begin{lemma}\label{unioning2}
Let $(B,M)$ be a local domain with $B/M$ uncountable, and let $R_0$ be a PN-subring of $B$.  Let $\Omega$ be a well-ordered set with least element $0$, and assume that for every $\alpha \in \Omega$, $|\{\beta \in \Omega \, | \, \beta < \alpha\}| < |B/M|$.  
%Let $\gamma(\alpha) = \sup\{\beta \in \Omega \, | \, \beta < \alpha\}$.  
Suppose $\{R_{\alpha} \, | \, \alpha \in \Omega\}$ is an ascending collection of rings such that if $\gamma(\alpha) = \alpha$, then $R_{\alpha} = \bigcup_{\beta < \alpha}R_{\beta}$ while if $\gamma(\alpha) < \alpha$, $R_{\alpha}$ is an PN-subring of $B$ with $R_{\gamma(\alpha)} \subseteq R_{\alpha}$ and prime elements in $R_{\gamma(\alpha)}$ are prime in $R_{\alpha}$.

Then $S = \bigcup_{\alpha \in \Omega} R_{\alpha}$ satisfies all conditions to be a PN-subring of $B$ except for possibly the condition that $|S| < |B/M|$.  Moreover, $|S| \leq \sup(|R_0|, |\Omega|)$ and elements that are prime in some $R_{\alpha}$ are prime in $S$.
\end{lemma}

\begin{proof}
The result follows from the proof of Lemma 6 in \cite{heitmannUFD}.
\end{proof}

We next show that we can construct a PN-subring of $B$ with many of our desired properties.

\begin{lemma}\label{together}
Let $(B,M)$ be a local domain with depth$B \geq 2$ and with $B/M$ uncountable, and let $(R, M \cap R)$ be a PN-subring of $B$. 
%Let $T = \widehat{B}$ and suppose that if $J \in \Min(T)$ then $T/J$ is a RLR.
%Suppose $Q$ is a prime ideal of $B$ with height at least two and 
Suppose $Q$ is a prime ideal of $B$ such that if $P \in \Spec(B)$ with $P \in \Ass(B/bB)$ for some $b \in B$, then $Q \not\subseteq P$.
Let $u \in B$.  Then there is a PN-subring $(S, M \cap S)$ of $B$ such that $R \subseteq S$, $|S| = |R|$, prime elements in $R$ are prime in $S$, $S$ contains an element of the coset $u + M^2$, $S$ contains a generating set for $Q$, and, for every finitely generated ideal $I$ of $S$, $IB \cap S = I$.
\end{lemma}

\begin{proof}
%Since $B$ has a prime ideal of height at least two, 
%Lemma \ref{heightone} implies that depth$B > 1$.  
By Lemma \ref{coset2}, there is a a PN-sburing $(R', M \cap R')$ of $B$ with $R \subseteq R'$, $|R'| = |R|$, prime elements in $R$ are prime in $R'$, and $R'$ contains an element of the coset $u + M^2$.  By Lemma \ref{generators2}, there exists a PN-subring $(R_0,M \cap R_0)$ of $B$ with $R' \subseteq R_0$ such that $|R_0| = |R'|$, prime elements in $R'$ are prime in $R_0$, and $R_0$ contains a generating set for $Q$.  Define
$$\Omega = \{(I,c) \, | \, I \mbox{ is a finitely generated ideal of } R_0 \mbox{ and } c \in IB \cap R_0\}.$$
Then $|\Omega| = |R_0| = |R|$.  Well-order $\Omega$, letting $0$ denote its first element.  We recursively define a family of PN-subrings as follows.  $R_0$ is already defined.  Suppose $\alpha \in \Omega$ and $R_{\beta}$ has been defined for all $\beta < \alpha$.  If $\gamma(\alpha) < \alpha$, then define $R_{\alpha}$ to be the subring obtained from Lemma \ref{closing} so that, if $\gamma(\alpha) = (I,c)$ then $R_{\alpha}$ is a PN-subring of $B$ with $R_{\gamma(\alpha)} \subseteq R_{\alpha}$, $|R_{\alpha}| = |R_{\gamma(\alpha)}|$, prime elements in $R_{\gamma(\alpha)}$ are prime in $R_{\alpha}$, and $c \in IR_{\alpha}$.  If, on the other hand, $\gamma(\alpha) = \alpha$, define $R_{\alpha} = \bigcup_{\beta < \alpha}R_{\beta}$.  Let $R_1 = \bigcup_{\alpha \in \Omega}R_{\alpha}$.  By Lemma \ref{unioning2}, $R_1$ is a PN-subring of $B$ with $|R_1| = |R_0|$ and elements that are prime in $R_0$ are prime in $R_1$.  In addition, by our construction, if $I$ is a finitely generated ideal of $R_0$ and $c \in IB \cap R_0$, then $c \in IR_1$ and so $IB \cap R_0 \subseteq IR_1$ for all finitely generated ideals $I$ of $R_0$.

Repeat this process replacing $R_0$ with $R_1$ to obtain a PN-subring $R_2$ such that $R_1 \subseteq R_2$, $|R_2| = |R_1|$, prime elements in $R_1$ are prime in $R_2$, and, if $I$ is a finitely generated ideal of $R_1$, then $IB \cap R_1 \subseteq IR_2$.  Continue to obtain an ascending chain of PN-subrings $R_0 \subseteq R_1 \subseteq R_2 \subseteq \cdots$ such that, for every $i \geq 0$, $|R_{i + 1}| = |R_i|$, prime elements in $R_i$ are prime in $R_{i + 1}$, and, if $I$ is a finitely generated ideal of $R_i$ then $IB \cap R_i \subseteq IR_{i + 1}$.  Let $S = \bigcup_{i = 1}^{\infty}R_i$.  By Lemma \ref{unioning2}, $S$ is a PN-subring of $B$ with $R \subseteq S$, $|S| = |R|$, and prime elements of $R$ are prime in $S$.  Also note that $S$ contains an element of the coset $u + M^2$ and $S$ contains a generating set for $Q$.  Now suppose that $I$ is a finitely generated ideal of $S$.  Then $I = (a_1, \ldots ,a_k)$ for some $a_i \in S$.  Let $c \in IB \cap S$. Then there is an $R_m$ such that $c \in R_m$ and $a_i \in R_m$ for every $i = 1,2, \ldots k$.  By construction, we have that $c \in (a_1, \ldots ,a_k)B \cap R_m \subseteq (a_1, \ldots ,a_k)R_{m + 1} \subseteq (a_1, \ldots ,a_k)S = I$.  It follows that $IB \cap S = I$ for every finitely generated ideal $I$ of $S$.
\end{proof}

We are now ready for the construction of our final UFD.

\begin{theorem}\label{FinalUFD}
Let $(B,M)$ be a local domain containing the rationals with depth$B \geq 2$, $B/M$ uncountable, and $|B| = |B/M|$.  
%Let $T = \widehat{B}$ and suppose that if $J \in \Min(T)$ then $T/J$ is a RLR.  
%Let $P$ be a height one prime ideal of $B$.  
Then there exists a quasi-local UFD $(A, M \cap A)$ such that
\begin{enumerate}
\item $A \subseteq B$,
%\item $P \cap A \neq (0)$,
\item The map $A \longrightarrow B/M^2$ is onto,
\item $IB \cap A = I$ for every finitely generated ideal $I$ of $A$, and
\item 
%If $Q \in \Spec(B)$ with ht$Q \geq 2$ then 
 $A$ contains a generating set for all $Q \in \Spec(B)$ satisfying the condition that if $P \in \Spec(B)$ with $P \in \Ass(B/bB)$ for some $b \in B$, then $Q \not\subseteq P$.
\end{enumerate}
\end{theorem}

\begin{proof}
Let $R_0=\mathbb{Q}$ and note that $R_0$ is a PN-subring of $B$.  
%Since $P$ is a height one prime ideal of an uncountable domain, it contains uncountably many elements.  Choose $p \in P$ such that $p$ is transcendental over $\mathbb{Q}$.  Then $R_0 = R[p]_{(R[p] \cap M)}$ is a PN-subring of $B$ with $R_0 \cap P \neq (0)$.  
Let $\Omega_1 = B/M^2$ and let $$\Omega_2 = \{Q \in \Spec(B) \, | \, \mbox{if } b \in B \mbox{ and } P \in \Ass(B/bB), \mbox{ then } Q \not\subseteq P\}.$$  Note that $|B/M^2| = |B/M|$ and $|\Omega_2| \leq |B| = |B/M|$.  Let $\Omega = \Omega_1 \times \Omega_2$ and observe that $|\Omega| = |B/M|$. 
%Suppose $|\Omega_2| < |B/M|$, and let $P$ be a height one prime ideal of $B$ that is contained in a height two prime ideal of $B$. Then dim$(B/P) \geq 2$ and the cardinality of the set of height one prime ideals of $B/P$ is less than $|B/M|$.  Letting $C$ be the set of height one prime ideals of $B/P$ and $D = \{0\}$, Lemma \ref{primeavoid} gives us that there is an $\overline{m} \in M/P$ such that $\overline{m}$ is not in a height one prime ideal of $B/P$, a contradiction. It follows that $|\Omega_1| = |\Omega_2|$.  
Well-order $\Omega$ using an index set $\Psi$ such that the first element of $\Psi$ is $0$ and such that each element of $\Omega$ has fewer than $|\Omega|$ predecessors. If $\alpha \in \Psi$, then denote the element in $\Omega$ that corresponds to $\alpha$ by $(b_{\alpha} + M^2, Q_{\alpha})$. 
%Let
%$$\Omega = \{(b_\alpha + M^2, Q_{\alpha}) \, | \, \alpha \in \Psi\}$$
%and note that $\Omega$ is well-ordered using the index set $\Psi$.

We now recursively define a family of PN-subrings $\{R_{\alpha} \, | \, \alpha \in \Psi\}$.  We have already defined $R_0$.  Let $\alpha \in \Psi$ and assume that $R_{\beta}$ has been defined for all $\beta < \alpha$.  If $\gamma(\alpha) < \alpha$ define $R_{\alpha}$ to be the PN-subring obtained from Lemma \ref{together} so that $R_{\gamma(\alpha)} \subseteq R_{\alpha}$, $|R_{\alpha}| = |R_{\gamma(\alpha)}|$, prime elements in $R_{\gamma(\alpha)}$ are prime in $R_{\alpha}$, $R_{\alpha}$ contains an element of the coset $b_{\gamma(\alpha)} + M^2$, $R_{\alpha}$ contains a generating set for $Q_{\gamma(\alpha)}$, and, for every finitely generated ideal $I$ of $R_{\alpha}$, $IB \cap R_{\alpha} = I$.  If $\gamma(\alpha) = \alpha$, define $R_{\alpha} = \bigcup_{\beta < \alpha}R_{\beta}$.  

Define $A = \bigcup_{\alpha \in \Psi}R_{\alpha}$.  By Lemma \ref{unioning2}, $(A, M \cap A)$ satisfies the conditions for being a PN-subring except for the condition that $|A| < |B/M|$.  In particular, $A$ is a UFD.  We now show that $IB \cap A = I$ for every finitely generated ideal $I$ of $A$. Let $I = (a_1,a_2, \ldots ,a_n)$ be a finitely generated ideal of $A$ and let $c \in IB \cap A$. Then there is an $\alpha \in \Psi$ such that $c, a_1, a_2, \ldots ,a_n \in R_{\alpha}$ and $JB \cap R_{\alpha} = J$ for every finitely generated ideal $J$ of $R_{\alpha}$.  Letting $J = (a_1, a_2, \ldots ,a_n)R_{\alpha}$, we have $c \in (a_1, a_2, \ldots ,a_n)B \cap R_{\alpha} = J \subseteq JA = I$, and it follows that $IB \cap A = I$.
By construction, $A$ satisfies the rest of the numbered conditions in the statement of the theorem.  
\end{proof}

Finally, we show that the UFD constructed in Theorem \ref{FinalUFD} satisfies our desired properties.

\begin{theorem}(The UFD Theorem)\label{FinalUFD2}
Let $(B,M)$ be a local domain containing the rationals with depth$B \geq 2$, $B/M$ uncountable, and $|B| = |B/M|$.  Let $0 \leq h \leq \mbox{dim}B - 2$.
%suppose that if $J \in \Min(T)$ then $T/J$ is a RLR.  
%Let $P$ be a height one prime ideal of $B$.  
Suppose that if $b \in B$  and $P \in \Ass(B/bB)$, then dim$(B/P) > h$.
Then there exists a local UFD $(A, M \cap A)$ such that
\begin{enumerate}
\item $A \subseteq B$,
\item $\widehat{A} = \widehat{B} = T$,
\item The map $f:\Spec(B) \longrightarrow \Spec(A)$ given by $f(P) = A \cap P$ is onto and, if $P$ is a prime ideal of $B$ with dim$(B/P) \leq h$ then $f(P)B = P$.  In particular, dim$(B/P) \leq h$ if and only if dim$(A/f(P)) \leq h$.  Moreover, if $P$ and $P'$ are prime ideals of $B$ with dim$(B/P) \leq h$ and dim$(B/P') \leq h$ then $f(P) = f(P')$ implies that $P = P'$. 
%\item $P \cap A \neq (0)$,
%\item The map $A \longrightarrow B/M^2$ is onto,
%\item $IB \cap A = I$ for every finitely generated ideal $I$ of $A$, and
%\item If $Q \in \Spec(B)$ with ht$Q \geq 2$ then $A$ contains a generating set for $Q$.
\end{enumerate}
\end{theorem}

\begin{proof}
Let $A$ be the UFD obtained from Theorem \ref{FinalUFD}.  By Proposition \ref{completionsame}, $A$ is Noetherian and $\widehat{A} = \widehat{B} = T$.  If $J \in \Spec(A)$, then there is a prime ideal $P$ of $T$ such that $P \cap A = J$.  Noting that $(B \cap P) \cap A = P \cap A = J$, we have that $f$ is onto. 

Now let $Q \in \Spec(B)$ such that dim$(B/Q) \leq h$.  We claim $A$ contains a generating set for $Q$. To see this, first suppose that there is a $b \in B$ and $P \in \Ass(B/bB)$ such that $Q \subseteq P$. By hypothesis, dim$(B/P) >h$ and so dim$(B/Q) > h$, a contradiction. 
%if $b \in B$ and $P \in \Ass(B/bB)$, Since $A$ satisfies condition (4) from Theorem \ref{FinalUFD},  $A$ contains a generating set for $P$. Suppose $r \in B$ and $J \in \Ass(B/rB)$ with $P \subseteq J$.  Since the completion of $B/rB$ is $T/rT$, there is a $Q \in \Spec(T)$ such that $Q \in \Ass(T/rT)$ and $JT \subseteq Q$.  Hence $PT \subseteq JT \subseteq Q$. Since $B$ is a domain, $r$ is a regular element of $T$.  By hypothesis, dim$(T/Q) > h$.  Now, $\widehat{B/P} \cong T/PT$ and so dim$(T/PT) = \mbox{dim}(B/P) \leq h$. But since $PT \subseteq Q$, we have dim$(T/PT) \geq \mbox{dim}(T/Q) > h$, a contradiction. 
Thus $Q$ is not contained in any such prime ideal $P$ of $B$. It follows by condition (4) of Theorem \ref{FinalUFD} that $A$ contains a generating set for $Q$ as claimed. Hence, $(A \cap Q)B = f(Q)B = Q$. Therefore, if $P$ and $P'$ are prime ideals of $B$ with dim$(B/P) \leq h$ and dim$(B/P') \leq h$ then $(A \cap P)B = P$ and $(A \cap P')B = P'$.  Thus if $f(P) = f(P')$ then $P = P'$.

It remains to show that dim$(B/P) \leq h$ if and only if dim$(A/f(P)) \leq h$.  Suppose dim$(B/P) \leq h$. Now,  the completion of $A/(A \cap P)$ is $T/(A \cap P)T$ and the completion of $B/P = B/(A \cap P)B$ is $T/(A \cap P)T.$  It follows that dim$(A/f(P)) = \mbox{dim}(A/(A \cap P)) = \mbox{dim}(T/(A \cap P)T) = \mbox{dim}(B/P) \leq h$.  Conversely, suppose dim$(A/f(P)) \leq h$.  Then dim$(B/P) = \mbox{dim}(T/PT) \leq \mbox{dim}(T/(A \cap P)T) = \mbox{dim}(A/f(P)) \leq h$. \end{proof}

The properties of $f$ ensure that it is onto, order-preserving and, when restricted to the prime ideals of $B$ with coheight at most $h$, is a poset isomorphism from the prime ideals of $B$ of coheight at most $h$ to the prime ideals of $A$ of coheight at most $h$.  In other words, the parts of the spectra of $B$ and $A$ of coheights at most $h$ are the same.  While we use The UFD Theorem to prove our main result, it is an interesting theorem in its own right.  Essentially, The UFD Theorem says that, given a local ring $B$ satisfying the hypothesis of the theorem, one can find a Noetherian UFD with the same completion as $B$ and such that the ``top'' part of the spectrum of $B$ and the spectrum of the UFD are the same.  

In Section \ref{GluingNodes}, we define precisely what we mean by growing and gluing in posets, and we show that every finite poset with one maximal element can be realized by applying our growing and gluing process to a poset with exactly one element.  Then, in Section \ref{MainResults}, we use results from Section \ref{Growing and Gluing} to construct a ring with nice properties such that the spectrum of the ring contains a saturated subposet that is an isomorphic copy of our given finite poset.  We then construct a ring that contains a saturated subposet that is an isomorphic copy of our given poset in which all elements of the spectrum of the ring that correspond to the given finite poset have coheight at most $h$.   Finally, we use Theorem \ref{FinalUFD2} to show that our given finite poset is isomorphic to a saturated subset of the spectrum of a Noetherian UFD.

\section{Gluing and Splitting Nodes in Partial Orders}\label{GluingNodes}

Given a finite poset $X$ that we wish to be a saturated subset of the spectrum of a Noetherian UFD or a quasi-excellent domain, we seek to systematically and carefully ``unravel" $X$ into increasingly simpler posets that we can more easily show embed as a saturated subset into the spectrum of a desirable ring. After making the poset simpler and obtaining an embedding of that poset into a smaller, nice ring, we then reverse the process to obtain a larger ring that contains $X$ as a saturated subset of its spectrum  using the tools of Sections 3 and 4. This unraveling process has two different flavors: ``splitting" height zero nodes and ``retracting" height zero nodes into certain height one nodes. For instance, if one were to read Figure \ref{GrowingGluingPicture1} from right to left, it would appear that certain nodes have been split and other nodes have been retracted. In the end, we show that if $X$ is any poset with a unique maximal node, then $X$ can be systematically unraveled and reduced all the way to a point, regardless of how complicated $X$ is (see Theorem \ref{reduceToPoint}). To talk about splitting nodes in posets, we find that it is simpler to first define the notion of ``gluing" nodes together in posets. In this way, if $Y$ is obtained from $X$ by gluing nodes together in $X,$ then we can view $X$ as a poset where certain nodes in $Y$ have been split. Similarly, one may view growing as a reversal of retracting in some fashion. It is the goal of this section to make these notions rigorous.

\subsection{Complete subsets and equivalence relations on finite partial orders} Let $(X, \le)$ be a poset. Recall from Definition \ref{completeDefinition} that if $S \subseteq X$ is a subposet of $X,$ we call $S$ a \textit{complete subset} of $X$ if for all $s, t \in S$ and $x \in X,$ if $s \le x \le t,$ then $x \in S.$  Fix a poset $(X, \le),$ and a complete subposet $S \subseteq X.$ Define $x \sim y$ if and only if $x = y$ or both $x$ and $y$ are in $S.$

Consider $X/\sim = \{[x]: x \in X\},$ where $[x]$ is the equivalence class containing $x,$ and define $g_S: X \to X/\sim$ as $g_S(x) = [x].$ Declare $[x] \le_{\sim} [y]$ if and only if $x \le_X y$ or there are $s_1, s_2 \in S$ such that $x \le_X s_1,$ and $s_2 \le_X y.$ It is a straightforward (albeit tedious) exercise to show that this order is well-defined. 

\begin{lem}\label{gluesaturated} With $g_S, X,$ and $\sim$ as above, if $S$ is a complete subset of $X,$ then $(X/\sim, \le_{\sim})$ is a poset. Moreover, if $h: X \to Z$ is any poset map such that $h$ is constant on $S,$ then there is a unique poset map $\vp: X/\sim \to Z$ such that $\vp g_S = h.$ \end{lem} 

\begin{proof} The relation $\le_{\sim}$ is clearly reflexive, so we need only show that it is anti-symmetric and transitive. 

(Anti-Symmetry) Suppose $[x] \le_{\sim} [y],$ and $[y] \le_{\sim} [x].$ If $x \le_X y,$ and $y \le_X x,$ then $x = y,$ so $[x] = [y].$ If $x \not \le_X y,$ and $y \le_X x,$ then there are $s_1, s_2 \in S$ such that $x \le_X s_1$ and $s_2 \le_X y.$ Then $$s_2 \le_X y \le_X s_1,$$ so $y \in S$ since $S$ is a complete subset of $X.$ Therefore, $x \in S,$ so $[x] = [y].$ The case where $x \le_X y$ and $y \not \le_X x$ is similar. Lastly, if $x \not \le_X y$ and $y \not \le_X x,$ then there are $s_1, s_2, s_3, s_4 \in S$ such that $x\le_X s_1, s_2 \le_X y \le_X s_3,$ and $s_4 \le_X x.$ Therefore, both $x, y \in S,$ so $[x] = [y].$ 

(Transitivity) Suppose $[x] \le_{\sim} [y],$ and $[y] \le_{\sim} [z].$ If $x \le_X y$ and $y \le_X z,$ then $x \le_X z.$ If $x \not \le_X y$ and $y \le_X z,$ then there are $s_1, s_2 \in S$ such that $x \le_X s_1,$ and $s_2 \le_X y \le_X z,$ so $[x] \le_{\sim} [z].$ The case where $x \le_X y,$ and $y \not \le_X z$ is similar. Finally, if neither $x \le_X y$ nor $y \le_X z,$ then there are $s_1, s_2, s_3, s_4 \in S$ such that $x \le_X s_1, s_2 \le_X y, y\le_X s_3,$ and $s_4 \le_X z,$ so $[x] \le_{\sim} [z].$ 

Let $h: X \to Z$ be any poset map. Define $\vp: X/\sim \to Z$ as $\vp([x]) = h(x).$ Now $\vp$ is well-defined because $h$ is constant on $S.$ If $[x] \le_{\sim} [y],$ and $x \not \le_X y,$ then there are $s, t \in S$ such that $x \le_X s$ and $t \le_X y.$ Therefore, $h(x) \le_Z h(s) = h(t) \le_Z h(y).$ By construction, $\vp g_S = h.$ If $\vp'$ is any other poset map such that $\vp' g_S = h,$ and $[x] \in X/\sim,$ then $[x] = g_S(x),$ so $$\vp'([x]) = \vp'g_S(x) = h(x) = \vp g_S(x) = \vp([x]).$$ So $\vp$ is unique. 
\end{proof} 

Recall that if $X$ is a poset, and $A \subseteq X,$ we say $A$ is an \textit{antichain} if for all $a, b \in A,$ $a \le b$ implies that $a = b.$

\begin{lem}\label{saturatedexamples} Let $(X, \le)$ be a poset. The following sets are complete subsets of $X:$

\begin{enumerate} 
\item Any antichain in $X.$ In particular, any subset of $\min X$ or $\max X.$ 
\item The set $g^{-1}(y),$ where $g: X \to Y$ is any poset map, and $y \in Y.$ 
\item $G_X(x) := \{u \in X: x \le_X u\},$ where $x \in X.$
\item $L_X(x) := \{u \in X: u \le_X x\},$ where $x \in X.$ 
\item The intersection of any collection of complete subsets of $X.$ 
\end{enumerate}
\end{lem} 

\begin{proof} We prove the second and fifth statements. The rest follow immediately from their definitions. For (2), if $c \le_X x \le_X d$ for $c, d \in g^{-1}(y),$ then $$y \le_Y g(x) \le_Y y,$$ so $g(x) = y,$ and $x \in g^{-1}(y).$ For statement (5), if $\{S_i\}_{i \in I}$ is a collection of complete subsets, and $c \le_X x \le_X d$ for $c$ and $d$ in $\cap_{i \in I} S_i,$ then $x \in \cap_{i \in I} S_i$ since each $S_i$ is complete. \end{proof} 

\subsection{Gluing nodes in a poset} 

\begin{mydef}\label{compatible} Let $X$ and $Y$ and be posets, and let $g: X \to Y$ be a surjective poset map. If $Z$ is a poset and $h: X \to Z$ is any poset map, we say $h$ is \textit{compatible} with $g$ if, for each $y \in Y,$ the restriction of $h$ to $g^{-1}(y)$ is constant. \end{mydef}

\begin{mydef}\label{gluingdef} Let $X$ be a poset. We say that the poset $Y$ is a \textit{gluing} of $X$ with gluing map $g: X \to Y$ if whenever $h: X \to Z$ is compatible with $g,$ there exists a unique poset map $\vp: Y \to Z$ such that $\vp g = h.$ Moreover, if $C \subseteq X$ is a complete subset of $X$ such that:
\begin{enumerate}
\item $g$ is constant on $C,$ and
\item if $g(x) = g(x')$ for distinct $x, x' \in X,$ then both $x, x' \in C,$
\end{enumerate}
then we say \textit{$Y$ is a gluing of $X$ along $C.$}

\end{mydef}

\begin{remark}\label{minimalC}
    Gluings along $C$ are unique up to poset isomorphism. If $g_1: X \to Y_1$ is one gluing along $C$ and $g_2: X \to Y_2$ is another, then $g_1$ and $g_2$ are compatible with each other, so there exist poset maps  $\vp_1: Y_1 \to Y_2$ and $\vp_2: Y_2 \to Y_1$ such that $\vp_1\vp_2 = \id_{Y_2}$ and $\vp_2\vp_1 = \id_{Y_1}.$ Every poset $X$ is a gluing of itself along $C = \emptyset$ with gluing map $g = \id_X.$ In particular, if $Y$ is a gluing of $X$ along $C = \emptyset,$ then $Y \cong X.$  
\end{remark}

\begin{remark}\label{nonGluingEx}
    If $X$ and $Z$ are posets and $g: X \to Y$ is a surjective poset map, one may test that $h: X \to Z$ is compatible with $g$ by showing that for all $x, y \in X,$ if $g(x) = g(y),$ then $h(x) = h(y).$ In fact, $h$ is compatible with $g$ if and only if $h$ factors uniquely through $g;$ that is, there is a unique function $\vp: Y \to Z$ such that $h = \vp g.$ However, even if $h$ is compatible with $g$ and both are poset maps, there is no guarantee that $\vp$ is a poset map. For instance, take $X =\{a, b\}$ to be the antichain with two points, $Y$ to be the chain with two points $\{c < d\}$, and $Z = X.$ Let $h = \text{id}_X$ and define $g(a) = c; g(b) = d.$ Then both $h$ and $g$ are poset maps and $h$ is compatible with $g,$ but $\vp$ is not a poset map since $c < d$ in $Y$ yet $a= \vp(c) \not \le \vp(d) = b$ in $Z.$ Notice also that $Y$ is not a gluing of $X$ with gluing map $g.$ For one, Remark \ref{minimalC} implies that $Y$ would have to be isomorphic to $X,$ which it is not. Moreover, it conflicts with what one would expect to obtain from gluing nodes in $X.$ The only gluings we should expect from an antichain with two points would be merging both points together or doing nothing at all. In summary, $Y$ is a gluing of $X$ with gluing map $g$ if and only if for all poset maps $h: X \to Z,$ $h$ factors uniquely through $g$ via some \textit{poset map} $\vp.$
\end{remark}

\begin{theorem}\label{gluesaturated2} If $(X, \le)$ is a poset and $S$ is a complete subset of $X,$ then there is a gluing $Y$ of $X$ along $S.$ Moreover, if $Z$ is a gluing of $X$ along $S$ via any gluing map $h: X \to Z,$ then whenever $h(x) \le_Z h(y),$ either $x \le_X y$ or there exist $s, t \in S$ such that $x \le_X s$ and $t \le_X y.$ \end{theorem} 
\begin{proof} The first part follows immediately from Lemma \ref{gluesaturated}, where $Y = X/\sim,$ and the gluing map is $g_S: X \to X/\sim.$ 

For the second part, note that $g_S$ is compatible with $h,$ so there exists a poset map $\vp: Z \to Y$ such that $\vp h = g_S.$ If $h(x)\le_Z h(y),$ then $g_S(x) \le_Y g_S(y),$ which implies the conclusion immediately by construction of $Y$ and $g_S.$\end{proof}

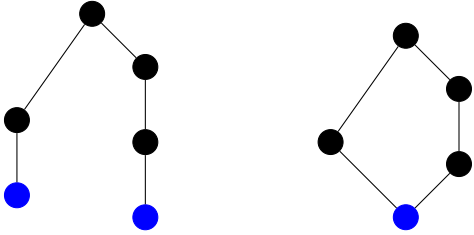
\begin{figure}[h] 
    \centering

\begin{tikzpicture}
%NODES
\node[main] (1) {};
\node[main] (4) [below right of=1] {};
\node[main] (5) [below of=4] {};
\node[main2] (6R) [below of= 5] {};
\node[main] (2) [below left of=4, left of=4] {};
\node[main2] (6L) [below of=2] {};

%DRAWING
\draw (1) -- (4) -- (5) -- (6R);
\draw (1) -- (2) -- (6L);

\end{tikzpicture}\hspace{2cm}
\begin{tikzpicture}
%NODES
\node[main] (1) {};
\node[main] (4) [below right of=1] {};
\node[main] (5) [below of=4] {};
\node[main2] (6) [below left of= 5] {};
\node[main] (2) [below left of=4, left of=4] {};

%DRAWING
\draw (1) -- (4) -- (5) -- (6) -- (2) -- (1);

\end{tikzpicture}
    \caption{A gluing example}
    \label{gluingFigure}
\end{figure}
 
\begin{example}

Consider Figure \ref{gluingFigure}.  
\noindent Let $X$ be the poset on the left, and let $S$ be the set of nodes that are colored blue in $X.$ $S$ is an antichain, so there is a gluing of $X$ along $S$ by Theorem \ref{gluesaturated2}. The poset $Y$ on the right is what results after gluing along $S.$ 
\end{example}

Reading Figure \ref{gluingFigure} from left to right shows how to glue along a subset of the left poset $X$ to obtain the right poset $Y.$ However, if one were to read the same figure from right to left, it would suggest that the left poset can be gotten by ``splitting" the unique blue node at the bottom of $Y$ into the two blue nodes at the bottom of $X.$ In fact, one may go further by then ``retracting" those two blue nodes in $X$ into the black nodes immediately above them.

This point-of-view will be crucial for our ultimate goal of embedding a finite poset into the spectrum of some Noetherian UFD or quasi-excellent domain. Indeed, given a finite poset $K_0$ with a single maximal node and such that $\dim K_0 \ge 1,$ we will find posets $K_1$ and $K_2$ such that $K_0$ is a gluing of $K_1$ along a subset of $\min K_1$ and  $K_2$ is gotten by systematically retracting certain minimal nodes of $K_1$ into their respective covers. Repeating this process as much as is necessary, we will eventually obtain a single point since our original poset $K_0$ has a single maximal node. Put another way, one can obtain any finite poset with a single maximal node by starting with a point, adding new minimal nodes, gluing a subset of them together in a suitable fashion, and repeating that process as necessary. Each such operation (e.g., adding new minimal nodes and gluing them) has a matching operation in the context of the spectrum of commutative rings, and this is what ultimately allows us to prove our main results.

\subsection{Gluing and splitting height zero nodes}

\begin{mydef}
    If $Y$ is a gluing of $X$ along a subset $C \subseteq \min X,$ we say $Y$ is a \textit{height zero gluing} of $X.$
\end{mydef}

We establish some preliminary and important results regarding height zero gluings.

\begin{lem}\label{dimensionsEqual}\label{coverRemark} If $Y$ is a height zero gluing of a finite poset $X$ along $C$ with gluing map $g,$ then the following statements are true:
\begin{enumerate}
    \item $g(\min X) = \min Y,$ 
    \item $\dim X = \dim Y,$ and
    \item if $g(x) <_c g(y),$ there exists $x' \in X$ such that $x' <_c y$ and $g(x') = g(x).$
\end{enumerate}
\end{lem}

\begin{proof} (1) Suppose $x \in \min X$ and $y \le g(x).$ Then $y = g(x')$ for $x' \in X.$ So $g(x') \le g(x)$ and by Theorem \ref{gluesaturated2}  either $x' \le x$ or there are $c_1, c_2 \in C$ such that $x' \le c_1$ and $c_2 \le x.$ If $x' \le x,$ then $x' = x$ since $x \in \min X.$ Hence $y = g(x).$ Otherwise, $x' = c_1$ and $c_2 = x.$ Since $c_1, c_2 \in C,$ we have $y = g(x') = g(c_1) = g(c_2) = g(x)$ still. Therefore, $g(x) \in \min Y$ and we have $g(\min X) \subseteq \min Y.$ Conversely, if $y \in \min Y,$ then $y = g(x)$ for some $x \in X.$ Since $X$ is finite, there is $x' \in \min X$ such that $x' \le x.$ Then $g(x') \le g(x)$ so $g(x') = g(x)$ since $g(x)$ is minimal. Therefore, either $x' = x$ or both $x', x \in C.$ In either case, $x$ is minimal. Therefore, $\min Y \subseteq g(\min X).$

(2) If $x < y$ in $X$ and $g(x) = g(y),$ then both $x, y \in C \subseteq \min X$ which is of course impossible. Therefore, $g(x) < g(y)$ and we see that $\dim X \le \dim Y.$ Conversely, suppose $g(x) < g(y).$ We claim that there is $x' \in X$ such that $g(x') = g(x)$ and $x' < y.$ Either $x < y$ or there exist $s, s' \in C$ such that $x \le s$ and $s' < y.$ In the former case, set $x' = x.$ In the latter case, note that $x = s$ since $s$ is minimal so that $g(x) = g(s) = g(s').$ In particular, we may set $x' = s'$ in this case. Therefore, every chain in $Y$ lifts to a chain of equivalent length in $X$ so $\dim Y \le \dim X$ hence $\dim Y = \dim X.$

(3) This follows from the proof of (2).
\end{proof}

\begin{mydef}
    If $Y$ is a finite poset, let $n(Y)$ denote the number of height zero nodes of $Y$ that have at least two covers. 
\end{mydef}

\begin{mydef}
    Let $Y$ be a finite poset. We say $X$ is a \textit{height zero splitting} of $Y$ if the following statements hold:
    \begin{enumerate}
        \item $Y$ is a height zero gluing of $X,$ and
        \item either $n(X) = 0$ or $n(X) < n(Y)$
    \end{enumerate}
\end{mydef}

We are thankful to the referee for suggesting the approach in the proof of the next lemma. It greatly simplified our original strategy for splitting nodes in finite posets. 

\begin{lem}\label{everythingSplits}
   Every finite poset has a height zero splitting.
\end{lem}

\begin{proof}
    Let $Y$ be a finite poset. If $n(Y) = 0,$ we may take $Y$ to be its own height zero splitting with gluing map the identity. Assume $n(Y) > 0,$ and let $x \in \min Y$ be an element with at least two covers. Let $A = \{(x, y) \in Y \times Y: x<_c y\}.$ Let $X:= Y\setminus \{x\} \cup A.$ We place an order relation on $X.$ Suppose $u, v \in X.$ If $u, v \in Y\setminus \{x\},$ then set $u \le_X v$ if and only if $u \le_Y v.$ If $u, v \in A,$ then declare $u\le_X v$ if and only if $u = v.$ Otherwise, put $u \le_X v$ if and only if $u = (x, y)$ for some $(x, y) \in A$ and $y \le_Y v.$ We claim $X$ has the desired properties. 
    
    If $(x,y) \in A$ and $(x, y) \le_X w,$ then $y \le_X w.$ So $y$ is the only cover of $(x, y)$ in $X.$ In particular, if $t \in \min X$ has at least two covers, then $t \in Y\setminus\{x\}.$ In fact, $t \in \min Y\setminus\{x\} \subsetneq \min Y.$ So $n(X) < n(Y).$ Also, $A \subseteq \min X$ since if $u \le_X v$ for $v \in A$ we must have $u \in A$ by definition of the order on $X.$ Therefore, $u = v.$

    Define $g:X \to Y$ as $g(z) = z$ if $z \in Y\setminus\{x\}$ and $g(x, y) = x$ if $(x, y) \in A.$ We claim $g$ is a gluing map along $C = g^{-1}(x) = A$. Note that $g$ satisfies (1) and (2) from Definition \ref{gluingdef}. Let $h: X \to Z$ be a poset map that is compatible with $g,$ and let $\vp: Y \to Z$ be the induced set map. Suppose $a \le_Y b.$ If $a = b,$ then of course $\vp(a) = \vp(b)$ by compatibility of $h.$ So assume $a <_Y b.$ If $a \ne x,$ then $b \ne x,$ so both $a, b \in X$ with $a <_X b$ and hence $h(a) \le_Z h(b).$ If $a = x,$ then since $Y$ is finite, there exists a cover $y$ of $x$ such that $x <_c y \le b.$ So $(x, y) \le_X b$ and $\vp(x) = h(x, y) \le_Z h(b) = \vp(b).$ Thus $\vp$ is a poset map. If $\vp'$ is another poset map from $Y$ to $Z$ such that $\vp'g = h,$ then we have $\vp'g = \vp g$ so that $\vp = \vp'$ since $g$ is surjective. \end{proof}

\begin{mydef}\label{retractionDefinition}
    Let $X$ be a poset such that $\dim X > 0.$ We say a height-one node $x$ is \textit{simple} if for all $u \in X$ such that $u < x,$ $x$ is the only node that covers $u.$ If $x$ is a simple node in $X,$ define $D_x:=\{u\in X: u< x\}.$ We say $X'$ is a \textit{retraction} of $X$ if there exists a simple node $x \in X$ such that $X' = X \setminus D_x$ with the same order relations that are on $X.$
\end{mydef}

\begin{remark}
    Although the definition does not include an explicit ``retraction" in the strictest sense of the word, there is a surjective poset map $r: X \to X'$ such that $r(u) = x$ for all $u \in D_x$ and $r(z)= z$ for all $z \in X$ outside $D_x.$ In this way, one can view the elements of $D_x$ as having been retracted into $x$ via $r.$ Yet another way to think of $X'$ is to start with $X$ and adjoin $D_x$ below $x$ (write $X = X' \cup D_x$ and recover the original order on $X$). In some sense, nodes have been ``grown" below the height zero node $x$ in $X'$ to form the original $X$ from $X'.$
\end{remark}

\begin{lem}\label{reduceDimension}
    If $X$ is a finite poset such that $n(X) = 0$ and $\dim X > 0,$ then there exists a sequence of posets $(X_1, \ldots, X_{m+1} = X)$ such that $X_i$ is a retraction of $X_{i+1}$ for all $1 \le i \le m,$ and $\dim X_1 < \dim X.$
\end{lem}
\begin{proof}
Let $x_1, \ldots, x_m$ be an enumeration of the height one nodes of $X.$ Since $n(X) = 0,$ each $x_i$ is simple. Let $X_{m+1} = X,$ set $X_m = X_{m+1}\setminus D_{x_m}.$ Having defined $X_i$ for some $1 < i \le m+1,$ let $X_i = X_{i+1}\setminus D_{x_i}.$ Now $X_1 = X \setminus \cup_{i=1}^m D_{x_i} \subset X.$ If $C$ is a chain in $X$ whose length is $\dim X,$ then $x_j \in C$ for some $j$ so $x_j$ has height zero in $X_j$ hence $X_1.$ Therefore, $\dim X_1 = \dim X - 1.$\end{proof}
\begin{remark}
    If $X'$ is a retraction of $X$ where $\dim X > 0,$ one may view $X$ as the union of $X'$ along with a set of new minimal nodes that have been added ``below" a certain minimal node of $X'.$ As another note, if $X' = X \setminus D_x,$ then $x$ has height zero (i.e. is minimal) in $X'.$
\end{remark}
\subsection{Reduction sequences} If we start with a nonempty finite poset $X,$ we may apply Lemma \ref{everythingSplits} repeatedly to eventually obtain a poset $X'$ with $n(X') = 0.$ Afterwards, we may repeatedly perform retractions, starting with $X',$ to eventually obtain a poset $X''$ such that $\dim X'' < \dim X.$ In this way, we have obtained a sequence of posets $(X'', \ldots, X', \ldots, X)$ where each term of the sequence is either a height zero splitting or a retraction of the next. In this scenario, we say $X$ has been \textit{reduced} to $X''$ and the sequence is a \textit{reduction sequence.} 

\begin{comment}
%If we start with a nonempty finite poset $X,$ it has a height zero splitting by Lemma \ref{everythingSplits}. By repeated application of \ref{everythingSplits} if necessary, we obtain a poset $X'$ with $n(X') = 0.$ 

%Since every minimal node in $X'$ has at most one cover, it has a retraction $X''$ as in Definition \ref{retractionDefinition}. Put another way, we get a sequence of posets $(X'', X', X)$ where $X'$ can be viewed as $X''$ along with a new set of minimal nodes that were added below height zero nodes of $X''$ and where $X$ is a height zero gluing of $X'.$ This is an example of what we call a \textit{reduction sequence,} and in this context we say $X$ can be \textit{reduced} to $X''.$

    \begin{mydef}
    A sequence $(X_1, \ldots, X_n)$ of posets is called a \textit{reduction sequence} if the following hold:
    \begin{enumerate}
        \item $n \equiv 0 \pmod 3,$
        \item given $X_i, X_{i+1}, X_{i+2}$ for $i \equiv 1 \pmod 3,$
        we have $X_{i+1}$ is a height zero splitting of $X_{i+2}$ and $X_i$ is the retraction of $X_{i+1},$ and
        \item if $i \equiv 0 \pmod 3$ for $1 \le i < n,$ then $X_i = X_{i+1}.$
    \end{enumerate}
    If $Y$ and $X$ are posets, we say $Y$ \textit{reduces to} $X$ if there exists a reduction sequence $(X_1, \ldots, X_n)$ such that $X_1 = X$ and $X_n = Y.$
\end{mydef}

\end{comment}

\begin{mydef}
    A sequence $(X_1, \ldots, X_n)$ of posets is called a \textit{reduction sequence} if for all $1\le i < n,$ $X_i$ is a height zero splitting of $X_{i+1},$ a retraction of $X_{i+1},$ or equal to $X_{i+1}.$ If $X$ and $Y$ are posets, we say $Y$ \textit{reduces to} $X$ if there exists a reduction sequence of the form $(X_1 = X, \ldots, X_n = Y).$
\end{mydef}

We are now ready to prove the main result of this section. It is stated in the context of posets with a unique maximal node since that is the focus of our main results.
\begin{theorem}\label{reduceToPoint}
    Every finite poset with a unique maximal node reduces to a point. 
\end{theorem}
\begin{proof}
    We induct on $\dim Y.$ If $\dim Y = 0,$ then $Y$ is a point and the sequence $(Y, Y)$ is a reduction sequence. Now suppose $\dim Y > 0$ and the assertion has been proved for all posets with a unique maximal node whose dimension is less than $\dim Y.$ By repeated application of Lemma \ref{everythingSplits}, there exists a reduction sequence of posets $(X, \ldots, Y)$ such that each predecessor is a height zero splitting of its successor and $n(X) = 0.$ Note $\dim X = \dim Y$ by Lemma \ref{dimensionsEqual}. By Lemma \ref{reduceDimension}, there exists a reduction sequence $(Z, \ldots, X)$ of posets such that each predecessor is a retraction of its successor and $\dim Z < \dim X = \dim Y.$ Since $Z$ has a unique maximal node, there exists, by the induction hypothesis, a reduction sequence of the form $(Z_1, \ldots, Z_m = Z)$ where $Z_1$ is a point. Therefore, $Y$ reduces to a point via the reduction sequence $(X_1, \ldots, X_r=Y)$ which is gotten by concatenating and relabeling the prior reduction sequences. \end{proof}

\section{The Main Results}\label{MainResults}

In this section we prove our two main results, Corollary \ref{FinalQuasiExcellent} and Theorem \ref{FinalUFDTheorem}.  Corollary \ref{FinalQuasiExcellent} states that every finite poset is isomorphic to a saturated subset of the spectrum of a quasi-excellent domain and Theorem \ref{FinalUFDTheorem} states that every finite poset is isomorphic to a saturated subset of the spectrum of a Noetherian UFD.  We begin with a preliminary lemma. Then, in Theorem \ref{elevateX}, we use Theorem \ref{onegrowing} to show that if $X$ is a retraction of the finite poset $Z$, and $B$ is a local ring (satisfying some nice properties) whose spectrum contains a saturated subset that is isomorphic to $X$, then there is a local ring $S$ (also satisfying some nice properties) whose spectrum contains a saturated subset that is isomorphic to $Z$.

\begin{lemma}\label{embedding}
Let $B$ be a Noetherian ring, let $y$ and $z$ be indeterminates, let $A = B[[y,z]]$, and let $X$ be a poset.  Suppose $\phi: X \longrightarrow \Spec(B)$ is a saturated embedding.  Then $\Psi: X \longrightarrow \Spec(A)$ given by $\Psi(u) = (\phi(u), y, z)A$ is a saturated embedding.
\end{lemma}

\begin{proof} Let $\pi$ be the map from $\Spec(B)$ to $\Spec(A)$ given by $\pi(P) = (P,y,z),$ and note that $\pi$ is a poset embedding. Since $A/(y,z) \cong B,$ the prime ideals of $A$ containing $(y,z)$ are in one-to-one-correspondence with the prime ideals of $B$. It follows that the image of $\pi$ is all prime ideals of $A$ that contain $(y,z)$. Hence, $\pi$ is a saturated embedding and so $\Psi = \pi \phi$ is also a saturated embedding. \end{proof}

When we write $|R| = c$ for a ring $R$, we mean that $R$ has the same cardinality as the set of real numbers.

\begin{theorem}\label{elevateX}
Let $X$ and $Z$ be finite posets, and suppose $X$ is a retraction of $Z.$ Let $B$ be a reduced local ring containing the rationals with $|B| = |B/M| = c.$ 
%Let $T = \widehat{B}$ be the completion of $B$ at its maximal ideal, and assume that if $I \in \Min{T},$ then $T/I$ is a RLR. 
Suppose $\vp: X \to \Spec (B)$ is a saturated embedding such that $\vp(\min X) = \Min{(B)}.$ Then there exists a local ring $S$ such that
\medskip
\begin{enumerate}
    \item $S$ is reduced,
    \item $S$ contains the rationals,
    \item If $N$ is the maximal ideal of $S,$ then $|S| = |S/N| = c,$
%    \item If $Q$ is a minimal prime ideal of $\widehat{S},$ then $\widehat{S}/Q$ is a RLR, and
    \item There is a saturated embedding $\psi: Z \to \Spec{(S)}$ such that $\psi(\min Z) = \Min{(S)}.$
\end{enumerate}

\noindent Moreover, if $B$ is quasi-excellent, then so is $S.$
\end{theorem}

\begin{proof}
    Since $X$ is a retraction of $Z,$ there exists a simple node $p \in Z$ such that $X = Z \setminus D_p.$ Enumerate $D_p$ as $q_1, \ldots, q_n.$ Recall $q_i <_c p$ in $Z$ for each $i,$ and $p \in X$ is minimal. Write $\Min(B) = \{\vp(p):= P_1, \ldots, P_m\}.$ Let $S = B[[y,z]]/J$ be as in the conclusion of Theorem \ref{onegrowing} applied to $\Min(B)$ and $n = |D_p|,$ and let $Q_1/J,\ldots, Q_n/J$ be the corresponding prime ideals of $S,$ each of which is constructed as in the proof of Theorem \ref{onegrowing}. Since $A = B[[y,z]]$ is local, so is $S$.  Since the ideal $J$ obtained from Theorem \ref{onegrowing} is the intersection of prime ideals, $S = A/J$ is reduced.  Since $B$ contains the rationals, so does $S$.  The maximal ideal of $S$ is $N = (M,y,z)$, and so $S/N \cong B/M$, and we have $|S/N| = |B/M| = c$.  Recall that $c^{\aleph_0} = c$ and so $|B| = |B[[y,z]]|$.  It follows that $|S| \leq |B[[y,z]]| = |B| = |B/M| = |S/N|$.  Since $|S/N| \leq |S|$, we have $|S| = |S/N| = c$. Therefore, $S$ satisfies parts $(1) - (3)$ of this lemma. 
    %and part (4) follows from part (3) of Theorem \ref{onegrowing}. 
    
    We now construct $\psi: Z \to \Spec(S).$ Let $x \in Z.$ If $x \in X,$ define $\psi(x):=(\vp(x), y, z)A/J.$ Otherwise, $x = q_i$ for some $i.$ In this case, we define $\psi(x) := Q_i/J.$ We claim $\psi$ has the desired properties. First note that if $w_1, w_2 \in Z$ with $w_1 \le_Z w_2$, then $\psi(w_1) \subseteq \psi(w_2).$  Now suppose $\psi(w_1) \subseteq \psi(w_2).$ Either $w_1 \in D_p$ or not. If $w_1 \notin D_p,$ then $\psi(w_1) = (\vp(w_1), y, z)A/J$ so that in particular $w_2 \notin D_p.$ Thus, $w_1, w_2 \in X$ and we have $$\psi(w_1) = (\vp(w_1), y, z)A/J \subseteq (\vp(w_2), y, z)A/J = \psi(w_2),$$ which implies $\vp(w_1) \subseteq \vp(w_2)$. Since $\vp$ is an embedding, $w_1 \le_X w_2$ hence $w_1 \le_Z w_2.$ If $w_1 \in D_p,$ then $w_1 = q_i$ for some $i,$ and we have $Q_i/J \subseteq \psi(w_2).$ If $\psi(w_2) = Q_i/J,$ then $w_1 = q_i = w_2$ by construction of $\psi.$ Otherwise, $w_2 \in X$ and we have $\psi(w_2) = (\vp(w_2), y, z)A/J.$ Each $Q_i = (\vp(p), y + \alpha z)A$ for some $\alpha \in \mathbb Q$ by construction. So $(\vp(p), y, z)A \subseteq (\vp(w_2), y, z)A,$ and therefore $\vp(p) \subseteq \vp(w_2)$. Since $\vp$ is an embedding, $p \le_Z w_2.$ Since $q_i <_Z p,$ we have $w_1 <_Z w_2.$ To show that $\psi$ is cover-preserving, we argue similarly. Suppose $w_1 <_c w_2$ in $Z.$ If $w_1 \notin D_p,$ then both $w_1, w_2 \in X$ so $\psi(w_1) <_c \psi(w_2)$ in $S$ by Lemma \ref{embedding} and properties of quotient rings. If $w_1 = q_i$ for some $i,$ then $w_2 = p$ because $p$ is the only cover of $q_i$ in $Z.$ By part (2) of Theorem \ref{onegrowing}, $(\vp(p), y, z)A/Q_i$ has height one in $A/Q_i,$ so $(\vp(p), y, z)A/J$ covers $Q_i/J$ in $S.$

    Lastly, if $\min X = \{p_1:=p, p_2, \ldots, p_m\},$ where $\vp(p_i) = P_i$ for all $1 \le i \le m,$ then $\min Z = \{q_1, \ldots, q_n, p_2, \ldots, p_m\}$ by definition of the retraction. Therefore, $$\psi(\min Z) = \{Q_1/J, \ldots, Q_n/J, (P_2, y, z)A/J, \ldots, (P_m, y, z)A/J\} = \Min(S).$$

    Finally, by Proposition \ref{growingexcellent}, if $B$ is quasi-excellent, then so is $S.$
\end{proof}

We now work to prove an analogous result for height zero gluings.  That is, if $Z$ is a nonempty finite poset, $Y$ a height zero gluing of $Z$, and $S$ a local ring (satisfying some nice properties) whose spectrum contains a saturated subset that is isomorphic to $Z$, we show in Theorem \ref{GlueZ} that there is a local ring $B$ (also satisfying some nice properties) whose spectrum contains a saturated subset that is isomorphic to $Y$. The main tool for our proof is Theorem \ref{biggluing}.  We start with a preliminary lemma.

\begin{lemma}\label{commonCompletion} Let $(S, M)$ and $(B, B \cap M)$ be two reduced local rings such that $B \subseteq S$ and $\widehat{B} = \widehat{S}.$ Suppose $f: \Spec{(S)} \to \Spec{(B)}$ given by $f(P) := P \cap B$ is surjective and satisfies $f(P)S = P$ for all prime ideals $P$ of $S$ whose height is positive. Then we have the following:
\begin{enumerate}
    \item If $Q$ is a minimal prime ideal of $S,$ then $Q \cap B$ is a minimal prime ideal of $B.$
    \item If $P_1 \cap B = P_2 \cap B$ for prime ideals $P_1, P_2$ of $S,$ then either $P_1 = P_2$ or both $P_1$ and $P_2$ are minimal prime ideals of $S.$ 
    \item If $P_1 \cap B \supseteq P_2 \cap B$ for prime ideals $P_1, P_2$ of $S,$ then there exists a prime ideal $Q$ of $S$ such that $P_1 \supseteq Q$ and $Q \cap B = P_2 \cap B.$
\end{enumerate}

\end{lemma}

\begin{proof} Let $T:=\widehat{S} = \widehat{B}.$

(1) Suppose $Q$ is a minimal prime ideal of $S$ and $Q \cap B \supseteq P,$ where $P$ is a prime ideal of $B.$ Let $Q'$ be a prime ideal of $T$ such that $Q' \cap S  = Q.$ Since $T$ is also the completion of $B,$ there exists a prime ideal $Q''$ of $T$ such that $Q' \supseteq Q''$ and $Q'' \cap B = P$ by the Going-Down Theorem. Therefore, $Q' \cap S =  Q =  Q'' \cap S$ because $Q$ is a minimal prime ideal of $S$ and $Q' \cap S \supseteq Q'' \cap S.$ Consequently, $Q\cap B = Q' \cap B = Q'' \cap B = P.$ So $Q \cap B$ is a minimal prime ideal of $B.$

(2) Suppose not both $P_1$ and $P_2$ are minimal prime ideals of $S.$ Then $\height P_1 > 0$ without loss of generality. Since $P_1 \cap B = P_2 \cap B,$ we have that $$P_1 = (P_1 \cap B)S = (P_2 \cap B)S \subseteq P_2.$$ So $\height P_2 > 0,$ and thus $P_2 = (P_2 \cap B)S$ so that $P_1 = P_2.$

(3) Choose a prime ideal $P_1'$ of $T$ such that $P_1' \cap S = P_1.$ Then $P_1' \cap B = P_1 \cap B.$ By the Going-Down Theorem, there exists a prime ideal $P_2'$ of $T$ such that $P_1' \supseteq P_2'$ and $P_2' \cap B = P_2 \cap B.$ Let $Q:= P_2' \cap S.$ Then $$P_1 = P_1' \cap S \supseteq P_2' \cap S = Q,$$ and $Q \cap B = P_2 \cap B.$
\end{proof}

\begin{theorem}\label{GlueZ} Let $Z$ be a nonempty finite poset, and let $Y$ be a height zero gluing of $Z.$ Let $(S, M)$ be a reduced local ring containing $\mathbb Q$ such that $S/M$ is uncountable and $|S| = |S/M|.$ Suppose $\psi: Z \to \Spec S$ is a saturated embedding such that $\psi(\min Z) = \Min{(S)}.$ Then there exists a reduced local ring $B \subseteq S$ with maximal ideal $B \cap M$ such that 
\begin{enumerate}
    \item $B$ contains $\mathbb Q,$
    \item $\widehat{B} = \widehat{S},$
    \item $B/(B \cap M)$ is uncountable, $|B| = |B/(B \cap M)|,$ and
    \item There is a saturated embedding $\vp:Y \to \Spec B$ such that $\vp(\min Y) = \Min{(B)}.$
\end{enumerate}\end{theorem} 
\noindent Moreover, if $S$ is quasi-excellent, then so is $B$.

\begin{proof}
    Let $C \subseteq \min Z$ be so that $Y$ is a gluing of $Z$ along $C$ with gluing map $g: Z \to Y.$ If $C$ is empty, then $Y = Z$ (after relabeling nodes if necessary - see Remark \ref{minimalC}) and we may set $B = S.$  Otherwise, assume $C$ is nonempty. If $C$ is a proper subset of $\min Z,$ let $\mathcal P = \{\psi(C)\}\cup\{\{\psi(x)\}: x\in \min Z\setminus C\}.$ Otherwise, let $\mathcal P = \{\psi(\min Z)\}.$ Note that $\mathcal P$ is a partition of $\Min(S).$ Let $B$ be the reduced local ring obtained from the conclusion of Theorem \ref{biggluing} applied to $\mathcal P.$ Let $f: \Spec S \to \Spec B$ be given by $f(P) = P \cap B,$ and let $h: Z \to \Spec B$ be defined as $h := f\psi.$ We claim $h$ is compatible with $g$ (see Definition \ref{compatible}). Indeed, if $g(a) = g(b)$ for $a \ne b,$ then $a, b \in C$ since $g$ is a gluing along $C.$ Therefore, $\psi(a), \psi(b) \in \psi(C),$ and since $\psi(C) \in \mathcal P,$ we have $h(a) = f\psi(a) = f\psi(b) = h(b)$ by part (4) of Theorem \ref{biggluing}. Therefore, there exists a poset map $\vp: Y \to \Spec B$ such that $\vp g = h = f\psi.$

    The ring $B$ satisfies $(1)-(3),$ so we need only show that $\vp$ is a saturated embedding satisfying $(4).$ In particular, we must show that if $\vp(g(x)) \subseteq \vp(g(y)),$ then $g(x) \le g(y)$ in $Y$ to demonstrate that $\vp$ is an embedding from $Y$ into $\Spec{(B)}.$ Suppose $\vp(g(x)) \subseteq \vp(g(y)).$ Then $$f\psi(x) = \psi(x) \cap B \subseteq  \psi(y) \cap B = f\psi(y).$$  Let $P_1 = \psi(x)$ and $P_2 = \psi(y).$ Then $P_1 \cap B \subseteq P_2 \cap B.$ If $\height P_1 \ne 0,$ then $P_1 \subseteq P_2$ by part (5) of Theorem \ref{biggluing}, and since $\psi$ is an embedding, we have $x \le_Z y$ in $Z$ so $g(x) \le_Y g(y).$ Otherwise, $\height P_1 = 0.$ By part (3) of Lemma \ref{commonCompletion}, there exists a prime ideal $Q$ of $S$ such that $P_2 \supseteq Q$ and $Q \cap B = P_1 \cap B.$ Therefore, both $Q$ and $P_1 \in \psi(D)$ where $D = C$ or $D$ is a singleton subset of $\min Z$ by definition of partition $\mathcal P.$ Write $Q = \psi(d)$ for some $d \in D.$  Then $g(x) = g(d)$ because both $x, d \in D,$ and $\psi(d) \subseteq \psi(y)$ implies that $d \le_Z y$ since $\psi$ is an embedding. Therefore, $g(x) = g(d) \le_Y g(y)$ and so $\vp$ is an embedding.

    To see that $\vp$ is saturated, suppose $g(x) <_c g(y)$ in $Y.$ By Lemma \ref{coverRemark}, there exists $x' \in Z$ such that $g(x') = g(x)$ and $x' <_c y.$ Hence $\psi(y)$ covers  $\psi(x')$ in $\Spec(S).$ In other words, the chain $\psi(x')\subsetneq \psi(y)$ is saturated. Write $P_2':= \psi(y)$ and $P_1':=\psi(x')$ in $\Spec(S).$ We claim $P_2' \cap B$ covers $P_1' \cap B$ in $\Spec(B).$ Suppose not.  Then there is a prime ideal $P$ of $B$ such that $P_2'\cap B \supsetneq P \supsetneq P_1' \cap B$.  Since $f$ is onto, there is a prime ideal $P'$ of $S$ such that $P = P' \cap B$.  So we have \begin{equation}\label{E} P_2'\cap B \supsetneq P' \cap B \supsetneq P_1' \cap B\end{equation}for some prime ideal $P'$ of $S.$ Note that by Lemma \ref{commonCompletion}, $\height P' \ne 0$ so that $P_2' \supseteq P'.$ 

Suppose that $\height P_1' = 0.$ By Lemma \ref{commonCompletion} there is a prime ideal $Q'$ of $S$ such that $P' \supseteq Q'$ and $Q' \cap B = P_1' \cap B.$ Then both $P_1'$ and $Q'$ belong to $\psi(E)$ where, as above, $E = C$ or $E$ is a singleton of $\min Z.$ So $Q' = \psi(e)$ for some $e \in E.$ Note that $x' \in E$ as well so that $g(x') = g(e).$ In particular, $\psi(e) \subsetneq \psi(y)$ in $\Spec{(S)},$ and thus $e < y$ in $Z.$ Now, $y$ cannot cover $e$ in $Z$ because $\psi$ is a saturated embedding and $\psi(y) = P_2'$ does not cover $\psi(e) = Q'$ in $\Spec{(S)}.$ So there exists $r \in Z$ such that $e < r < y.$ Note that $g(r) \ne g(e)$ because $r$ is not minimal in $Z,$ and likewise $g(y) \ne g(r).$ But $$g(x) = g(x') = g(e) < g(r) < g(y),$$ a contradiction because $g(y)$ covers $g(x)$ in $Y.$ Therefore, $\height P_1' \ne 0,$ and since $\height P_2' \ne 0$ and $\height P' \ne 0,$ we have, upon extending (\ref{E}) to $S,$ that $P_2' \supseteq P' \supseteq P_1'.$ So $P' = P_2'$ or $P' = P_1',$ and thus either $P_2' \cap B = P' \cap B$ or $P' \cap B = P_1' \cap B,$ which violates (\ref{E}). So there are no prime ideals of $B$ strictly between $P_2' \cap B$ and $P_1' \cap B.$ That is, $P_2' \cap B$ covers $P_1' \cap B$ in $\Spec{(B)}.$ In summary, since $y$ covers $x'$ in $Z,$ $\psi(y)$ covers $\psi(x')$ in $\Spec{(S)},$ and we have just shown that $f\psi(y)$ must cover $f\psi(x')$ in $\Spec{(B)}.$ In particular, $\vp g(y) = f\psi(y)$ covers $f\psi(x') = \vp g(x') = \vp g(x).$ So $\vp$ is a saturated embedding.

To see that $\vp(\min Y) = \Min(B),$ note that $g(\min Z) = \min Y$ by Lemma \ref{dimensionsEqual}, so $$\vp(\min Y) = \vp g(\min Z) = f\psi (\min Z) = f(\Min(S)) = \Min(B).$$ Finally, by Proposition \ref{excellentgluing}, if $S$ is quasi-excellent, then so is $B$.\end{proof}

We now state and prove a theorem that is crucial for both of our main results.  In particular, we show that if $X$ is a finite poset, then there exists a local domain (satisfying some nice properties) whose spectrum contains a saturated subset that is isomorphic to the given poset $X$.

\begin{theorem}\label{GoodRingB}
If $X$ is a finite poset, then there exists a local domain $(B, M)$ and a saturated embedding $\vp: X \to \Spec (B)$ such that:
\begin{enumerate}
    \item $B$ contains $\mathbb Q,$
    \item $|B| = |B/M| = c,$
    \item $B$ is quasi-excellent.
\end{enumerate}
\end{theorem}

\begin{proof}
Let $K$ be a finite poset such that $X$ is a saturated subset of $K$ and $K$ has both a single maximal node and a single minimal node. By Theorem \ref{reduceToPoint}, $K$ reduces to a point, so there exists a reduction sequence $(K_1, \ldots, K_n := K)$ where $K_1$ is a point and each predecessor is either a height zero splitting, retraction, or equal to its successor. Let $B_1 = \mathbb C$ and let $\vp_1: K_1 \to \Spec(B_1)$ be the unique poset map from $K_1$ onto $\Spec(B_1).$  Let $1 \le i < n$ and assume that a reduced local ring $(B_i, M_i)$ has been constructed so that there is a saturated embedding $\vp_i: K_i \to \Spec(B_i)$ for which the following statements hold:
\begin{enumerate}
    \item $B_i$ contains $\mathbb Q,$
    \item $|B_i| = |B_i/M_i| = c,$
    \item $\vp_i(\min K_i) = \Min(B_i),$
%    \item $\widehat{B_i}/I$ is a RLR for every $I \in \Min(\widehat{B_i}),$ and
    \item $B_i$ is quasi-excellent.
\end{enumerate}
Note that $B_1$ is a reduced local ring and that $B_1$ and $\vp_1$ satisfy properties (1)-(4).
We now construct a reduced local ring $(B_{i+1}, M_{i+1})$ and a saturated embedding $\vp_{i+1}: K_{i+1} \to \Spec(B_{i+1})$ satisfying properties (1)-(4). 

If $K_{i+1} = K_i,$ then set $B_{i+1} = B_i$ and $\vp_{i+1} = \vp_i.$ If $K_i$ is a height zero splitting of $K_{i+1},$ then $K_{i+1}$ is a height zero gluing of $K_i.$ Apply Theorem \ref{GlueZ} to $Y = K_{i+1},$ $Z = K_i,$ $\psi = \vp_i,$ and $S = B_i$ to obtain a reduced local ring $(B_{i+1}, M_{i+1})$ and saturated embedding $\vp_{i+1}: K_{i+1} \to \Spec(B_{i+1}).$ Properties $(1), (3),$ and $(4)$ follow immediately from the conclusion of Theorem \ref{GlueZ}. 
%Property (4) follows from the fact that $\widehat{B_i} = \widehat{B_{i+1}}.$ 
In addition, we have $|B_{i + 1}| =|B_{i+1}/M_{i+1}|$ and $\widehat{B_{i + 1}} = \widehat{B_i}$, from the conclusion of Theorem \ref{GlueZ}.  Thus, $|B_{i + 1}| =|B_{i+1}/M_{i+1}| = |\widehat{B_{i+1}}/M_{i+1}\widehat{B_{i + 1}}|=|\widehat{B_i}/M_i\widehat{B_i}|=|B_i/M_i| = c,$ and so property (2) holds as well. If $K_i$ is a retraction of $K_{i+1},$ apply Theorem \ref{elevateX} to $X = K_i,$ $Z = K_{i+1},$ $\vp = \vp_i,$ and the ring $B_i$ to obtain a reduced local ring $(B_{i+1}, M_{i+1})$ and saturated embedding $\vp_{i+1}: K_{i+1} \to \Spec(B_{i+1}).$ That $\vp_{i+1}$ and $B_{i+1}$ satisfy properties (1)-(4) is immediate from the statement of Theorem \ref{elevateX}.

Set $B:=B_n$ and $\vp:=\vp_n.$ Now $\vp: K \to \Spec(B)$ is a saturated embedding and the minimal prime ideals of $B$ are in one-to-one correspondence with $\min K.$ Therefore, $B$ is a reduced local ring with a single minimal prime ideal and is hence a domain. 
%Since $\dim K \ge 3$ and $\vp$ is an embedding, it follows that $\dim B \ge 3.$ If $x \in X,$ then $\height_K x \ge 2$ so $\height \vp(x) \ge 2$ in $\Spec(B)$ since $\vp$ is an embedding. 
That $B$ satisfies the other asserted properties is immediate. \end{proof}

Our first main result, Corollary \ref{FinalQuasiExcellent}, follows directly from Theorem \ref{GoodRingB} and Lemma \ref{isomorphic}.

\begin{corollary}\label{FinalQuasiExcellent}
Every finite poset is isomorphic to a saturated subset of the spectrum of a quasi-excellent domain.
\end{corollary}

\begin{proof}
The result follows from Theorem \ref{GoodRingB} and Lemma \ref{isomorphic}.
\end{proof}

We are now ready to use Theorem \ref{GoodRingB} and Theorem \ref{FinalUFD2} to prove our second main result.

\begin{theorem}\label{FinalUFDTheorem}
Every finite poset is isomorphic to a saturated subset of the spectrum of a Noetherian unique factorization domain.
\end{theorem}

\begin{proof}
Let $X$ be a finite poset. By Theorem \ref{GoodRingB}, there is a local domain $(B', M')$ and a saturated embedding $\vp: X \to \Spec (B')$ such that $B'$ contains the rationals and $|B'| = |B'/M'| = c$. Let dim$B' = m$, and let $n = m + 2.$ Define $B = B'[[x_1, \ldots ,x_n]]$ where $x_1, \ldots ,x_n$ are indeterminates.  Then $B$ is a local domain with maximal ideal $M = (M',x_1,\ldots ,x_n)$, $B$ contains the rationals, depth$B \geq 2$ and $|B|=|B/M| = c$. We also have dim$B = m + n$ and so $0 \leq m \leq \mbox{dim}B - 2$. Define $\pi: \Spec(B') \longrightarrow \Spec(B)$ by $\pi(P) = (P,x_1, \ldots,x_n)$. By the argument used in the proof of Lemma \ref{embedding}, $\pi$ is a saturated embedding and so $\Psi := \pi \vp$ is a saturated embedding from $X$ to $\Spec(B)$. If $z \in X$, then dim$(B/\Psi(z)) \leq \mbox{ dim}(B/(x_1, \ldots,x_n)) = \mbox{ dim}B' = m.$

By Theorem \ref{growcoheight}, if $f \in B'[[x_1, \ldots ,x_n]] = B$ and $Q \in \Spec(B)$ with $Q \in \Ass(B/fB)$, then dim$(B/Q) \geq n - 1 > m$. By Theorem \ref{FinalUFD2} with $h = m$, there is a local (Noetherian) UFD $(A, M \cap A)$ contained in $B$ such that $\widehat{A} = \widehat{B}$.  Moreover, 
%if $Q$ is a prime ideal of $B$ with $\height Q \ge 2,$ then $A$ contains a generating set for $Q.$ 
the map $f: \Spec{(B)} \to \Spec{(A)}$ given by $f(P) = P \cap A$ is onto and, if $P$ is a prime ideal of $B$ with dim$(B/P) \leq m$ then $f(P)B = P$. Define $\psi: X \to \Spec{(A)}$ by $\psi(x) = f(\Psi(x))$. Since $\Psi$ is a saturated embedding mapping elements of $X$ to prime ideals of $B$ with coheight at most $m$, and since $f$ is a poset isomorphism from the prime ideals of $B$ of coheight at most $m$ to the prime ideals of $A$ of coheight at most $m$, $\psi$ is a saturated embedding from $X$ to $\Spec(A)$.
Lemma \ref{isomorphic} then implies that $X$ is isomorphic to a saturated subset of the spectrum of the Noetherian UFD $A$.
%restricts to a saturated embedding from $X$ into $\Spec{(A)}.$ [[Add something about $\vp$]]
\end{proof}

\section*{Acknowledgements}

We are grateful to the referee for their many useful suggestions and insights.  In particular, they suggested a technique that greatly simplified our original strategy for the material covered in Section \ref{GluingNodes}. In addition, we are especially thankful to the referee for identifying an issue in Section \ref{UFDTheorem} that we have since resolved.

\end{document}